\documentclass[11pt]{article}
\usepackage{e-jc}
\usepackage[margin=1cm]{caption}
\usepackage{amsmath, amsthm, amssymb}

\usepackage{verbatim}
\usepackage{url}
\usepackage{tikz,tkz-graph}
\usepackage{subfig}
\usepackage{hyperref}

\theoremstyle{plain}
\newtheorem{thm}{Theorem}
\newtheorem*{lemmaA}{Lemma A}
\newtheorem*{mainfirst}{Theorem 1}

\newtheorem{lem}[thm]{Lemma}

\newtheorem*{lemma8}{Lemma 8}
\newtheorem*{lemma9}{Lemma 9}
\newtheorem*{lemma17}{Lemma 17}
\newtheorem*{lemma18}{Lemma 18}
\newtheorem*{lemma19}{Lemma 19}
\newtheorem*{lemma20}{Lemma 20}
\newtheorem*{lemma21}{Lemma 21}
\newtheorem*{lemma22}{Lemma 22}
\newtheorem*{lemma23}{Lemma 23}

\theoremstyle{definition}

\theoremstyle{remark}

\newcommand{\chif}{\chi_{f}}
\newcommand{\ch}{{\textrm{ch}}}
\newcommand{\fancy}[1]{\mathcal{#1}}

\newcommand{\G}{\fancy{G}}
\newcommand{\I}{\fancy{I}}

\newcommand{\floor}[1]{\left\lfloor#1\right\rfloor}
\newcommand{\ceil}[1]{\left\lceil#1\right\rceil}
\newcommand{\card}[1]{\left|#1\right|}
\newcommand{\set}[1]{\left\{ #1 \right\}}
\newcommand{\parens}[1]{\left( #1 \right)}
\newcommand{\brackets}[1]{\left[ #1 \right]}
\newcommand{\nonadj}{\not\leftrightarrow}
\newcommand{\adj}{\leftrightarrow}

\date{\dateline{June 6, 2015}{August 15, 2016}\\
\small Mathematics Subject Classifications: 05C69, 05C10}

\begin{document}
\title{Planar graphs have independence ratio at least $3/13$}
\author{Daniel W. Cranston\thanks{Department of Mathematics and Applied
Mathematics, Virginia Commonwealth University, Richmond, VA, 23284. email:
\texttt{dcranston@vcu.edu}} \and 
Landon Rabern\thanks{LBD Data Solutions, Lancaster, PA, 17601.  email:
\texttt{landon.rabern@gmail.com}. 
}}
%
\maketitle
\date

\begin{abstract}
The 4 Color Theorem (4CT) implies that every
$n$-vertex planar graph has an independent set of size at least $\frac{n}4$; this
is best possible, as shown by the disjoint union of many copies of $K_4$.
In 1968, Erd\H{o}s
asked whether this bound on 
independence number could be proved more easily than the full 4CT.  In 1976
Albertson
showed (independently of the 4CT)
that every $n$-vertex planar graph has an independent set of size at least
$\frac{2n}9$.  Until now, this remained the best bound independent of the 4CT.
Our main result improves this bound to $\frac{3n}{13}$.
\end{abstract}

\section{Introduction}

An \emph{independent set} is a subset of vertices that induce no edges.  The
independence number $\alpha(G)$ of a graph $G$ is the size of a largest
independent set in $G$.  Determining the independence number of an arbitrary
graph $G$ is widely-studied and well-known to be NP-complete.  In fact, this
problem remains NP-complete, even when restricted to planar graphs of maximum
degree 3 (see, for example, \cite[Lemma 1]{GareyJohnson}).
Thus, much work
in this area focuses on proving lower bounds for the independence number of
some special class of graphs, often in terms of $|V(G)|$.
The \emph{independence ratio} of a graph $G$ is the quantity
$\frac{\alpha(G)}{|V(G)|}$.  


An immediate consequence of the 4 Color Theorem~\cite{AppelHaken1,AppelHaken2}
is that every planar graph has independence ratio at least $\frac14$;
simply take the largest color class.  In fact, this bound is best possible, as
shown by the disjoint union of many copies of $K_4$.
In 1968, Erd\H{o}s~\cite{berge} suggested that perhaps this corollary could be
proved more easily than the full 4 Color Theorem.  And in 1976,
Albertson~\cite{Albertson} showed (independently of the 4~Color Theorem)
that every planar graph has independence ratio at least
$\frac29$.  Our main theorem improves this bound to $\frac3{13}$.

\begin{thm}
Every planar graph has independent ratio at least $\frac{3}{13}$.
\label{mainthm2}
\end{thm}

The proof of Theorem~\ref{mainthm2} is heavily influenced by Albertson's proof.
One apparent difference is that our proof uses the discharging method, while his
does not.  However, this distinction is largely cosmetic.  To demonstrate this
point, we begin with a short discharging version of the final
step in Albertson's proof, which he verified using edge-counting. 
Although the arguments are essentially equivalent, the discharging method is
somewhat more flexible.  In part it was this added flexibility that allowed us
to push his ideas further.

The proof of our main result has the following outline.  The bulk of the work
consists in showing that certain configurations are \emph{reducible}, i.e.,
they cannot appear in a minimal counterexample to the theorem.  
The remainder of the proof is a counting argument (called \emph{discharging}),
where we show that every planar graph contains one of the forbidden
configurations; hence, it is not a minimal counterexample.



In the discharging section, we give each vertex $v$ initial charge $d(v)-6$,
where $d(v)$ is the degree of $v$.  By Euler's formula the sum of the initial
charges is $-12$.  Our goal is to redistribute charge, without changing the
sum, (assuming that $G$ contains no reducible configuration) so that every
vertex finishes with nonnegative charge.  This contradiction proves that, in
fact, $G$ must contain a reducible configuration.
To this end, we want to show that $G$ contains a reducible configuration 
whenever it has many vertices of degree at most 6
near each other, since vertices of degree 5 will need to receive charge and
vertices of degree 6 will have no spare charge to give away.  (We will see in
Lemma~\ref{CrunchToOneVertexNeighborhoodSize} that $G$ must have minimum degree
5.) Most of the work in the reducibility section goes into proving various
formalizations of this intuition.  

Typically, proofs like ours present the reducibility portion before
the discharging portion. 
However, because many of our reducibility arguments are quite technical, we
make the unusual choice to give the discharging first, with the goal of
providing context for the reducible configurations.  (Usually the process of
finding a proof switches back and forth between discharging and reducibility.
By necessity, though, the proof must present one of these first.)

We start with definitions.
A \emph{$k$-vertex} is a vertex of degree $k$; similarly, a $k^-$-vertex
(resp.~$k^+$-vertex) has degree at most (resp.~at least) $k$.
A $k$-neighbor of a vertex $v$ is a $k$-vertex that is a neighbor of $v$; and
$k^-$-neighbors and $k^+$-neigbors are defined analogously.
A $k$-cycle is a cycle of length $k$.
A vertex set $V_1$ in a connected graph $G$ is \emph{separating} if $G\setminus V_1$
has at least two components.  A cycle $C$ is separating if $V(C)$ is separating.
An \emph{independent $k$-set} is an independent set of size $k$.
When vertices $u$ and $v$ are adjacent, we write $u\adj v$; otherwise $u\nonadj
v$.

For a vertex $v$, let $H_v$ denote the subgraph induced by the $5$-neighbors and
$6$-neighbors of $v$.  Throughout the proof we consider a (hypothetical) minimal
counterexample $G$, which will be a triangulation.  In
Lemma~\ref{alphaNoTriangles}, we show that $G$ has no separating 3-cycle. 
These properties together imply that, for every vertex $v$, the subgraph
induced by the neighbors of $v$ is a cycle.  If some $w\in V(H_v)$ has
$d_{H_v}(w)=0$, then $w$ is an \emph{isolated} neighbor of $v$; otherwise $w$
is a \emph{non-isolated}
neighbor.  A non-isolated 5-neighbor of a vertex $v$ is \emph{crowded} (with
respect to $v$) if it has two $6$-neighbors in $H_v$.  We use crowded 5-neighbors
in the discharging proof to help ensure that $7$-vertices finish with sufficient
charge, specifically to handle the configuration in Figure~\ref{fig:whycrowded}.

\begin{figure}[!bh]
\centering
\begin{tikzpicture}[scale = 9]
\tikzstyle{VertexStyle}=[shape = circle, minimum size = 6pt, inner sep = 1.2pt, draw]
\Vertex[x = 0.60, y = 0.70, L = \small {$v$}]{v0}
\Vertex[x = 0.50, y = 0.90, L = \small {$6$}]{v1}
\Vertex[x = 0.70, y = 0.90, L = \small {$5$}]{v2}
\Vertex[x = 0.80, y = 0.75, L = \small {$6$}]{v3}
\Vertex[x = 0.40, y = 0.75, L = \small {$7^+$}]{v4}
\Vertex[x = 0.45, y = 0.55, L = \small {$6$}]{v5}
\Vertex[x = 0.60, y = 0.50, L = \small {$6$}]{v6}
\Vertex[x = 0.75, y = 0.55, L = \small {$7^+$}]{v7}
\Edge[](v1)(v0)
\Edge[](v2)(v0)
\Edge[](v3)(v0)
\Edge[](v4)(v0)
\Edge[](v5)(v0)
\Edge[](v6)(v0)
\Edge[](v7)(v0)
\Edge[](v5)(v6)
\Edge[](v5)(v4)
\Edge[](v1)(v4)
\Edge[](v1)(v2)
\Edge[](v2)(v3)
\Edge[](v3)(v7)
\Edge[](v7)(v6)
\end{tikzpicture}
\caption{A 7-vertex $v$ gives no charge to any crowded 5-neighbor.}
\label{fig:whycrowded}
\end{figure}
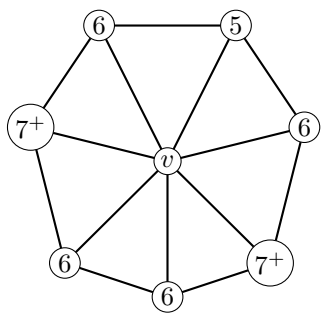

\section{Discharging: A Warmup}
As a warmup to our main proof, in this section
we give a short discharging proof that every planar triangulation with minimum
degree 5 and no separating 3-cycle must contain a certain configuration, which
Albertson showed could not appear in a minimal planar graph with independence
ratio less than $\frac29$.
(In fact, finding this proof helped encourage us to begin work on the present
paper.)

\begin{lemmaA}
Let $u$ and $v$ be adjacent vertices, such that $uvw$ and $uvx$ are 3-faces and
$d(w)=5$ and $d(x)\le 6$; call this configuration $H$.  (See Figure~\ref{figH}.)
If $G$ is a plane triangulation with minimum degree 5 and no separating
3-cycle, then $G$ contains a copy of $H$.
\end{lemmaA}
\begin{proof}
\begin{figure}
\subfloat[Adjacent vertices $u$ and $v$, with nonadjacent common 5-neighbors $w$
and $x$.]{\makebox[.5\textwidth]{
\begin{tikzpicture}[scale = 13, rotate=90]
\tikzstyle{VertexStyle} = []
\tikzstyle{EdgeStyle} = []
\tikzstyle{labeledStyle}=[shape = circle, minimum size = 6pt, inner sep = 1.2pt, draw]
\tikzstyle{unlabeledStyle}=[shape = circle, minimum size = 6pt, inner sep = 1.2pt, draw]
\Vertex[style = unlabeledStyle, x = 0.60, y = 0.70, L = \small {$v$}]{v0}
\Vertex[style = unlabeledStyle, x = 0.70, y = 0.70, L = \small {$u$}]{v1}
\Vertex[style = unlabeledStyle, x = 0.65, y = 0.80, L = \small {$w$}]{v2}
\Vertex[style = unlabeledStyle, x = 0.65, y = 0.60, L = \small {$x$}]{v3}
\Vertex[style = unlabeledStyle, x = 0.75, y = 0.80, L = \small {}]{v4}
\Vertex[style = unlabeledStyle, x = 0.55, y = 0.80, L = \small {}]{v5}
\Vertex[style = unlabeledStyle, x = 0.65, y = 0.90, L = \small {}]{v6}
\Vertex[style = unlabeledStyle, x = 0.75, y = 0.60, L = \small {}]{v7}
\Vertex[style = unlabeledStyle, x = 0.55, y = 0.60, L = \small {}]{v8}
\Vertex[style = unlabeledStyle, x = 0.65, y = 0.50, L = \small {}]{v9}
\Edge[](v0)(v3)
\Edge[](v1)(v3)
\Edge[](v7)(v3)
\Edge[](v8)(v3)
\Edge[](v9)(v3)
\Edge[](v0)(v2)
\Edge[](v1)(v2)
\Edge[](v4)(v2)
\Edge[](v5)(v2)
\Edge[](v6)(v2)
\Edge[](v4)(v1)
\Edge[](v7)(v1)
\Edge[](v4)(v6)
\Edge[](v0)(v5)
\Edge[](v1)(v0)
\Edge[](v8)(v0)
\Edge[](v8)(v9)
\Edge[](v7)(v9)
\Edge[](v6)(v5)
\end{tikzpicture}
}}
\subfloat[Adjacent vertices $u$ and $v$, with nonadjacent common neighbors $w$
and $x$, of degree 5 and 6.]{\makebox[.5\textwidth]{
\begin{tikzpicture}[scale = 13, rotate=-90]
\tikzstyle{VertexStyle} = []
\tikzstyle{EdgeStyle} = []
\tikzstyle{labeledStyle}=[shape = circle, minimum size = 6pt, inner sep = 1.2pt, draw]
\tikzstyle{unlabeledStyle}=[shape = circle, minimum size = 6pt, inner sep = 1.2pt, draw]
\Vertex[style = unlabeledStyle, x = 0.60, y = 0.70, L = \small {$u$}]{v0}
\Vertex[style = unlabeledStyle, x = 0.70, y = 0.70, L = \small {$v$}]{v1}
\Vertex[style = unlabeledStyle, x = 0.65, y = 0.80, L = \small {$x$}]{v2}
\Vertex[style = unlabeledStyle, x = 0.65, y = 0.60, L = \small {$w$}]{v3}
\Vertex[style = unlabeledStyle, x = 0.75, y = 0.80, L = \small {}]{v4}
\Vertex[style = unlabeledStyle, x = 0.55, y = 0.80, L = \small {}]{v5}
\Vertex[style = unlabeledStyle, x = 0.60, y = 0.90, L = \small {}]{v6}
\Vertex[style = unlabeledStyle, x = 0.70, y = 0.90, L = \small {}]{v7}
\Vertex[style = unlabeledStyle, x = 0.75, y = 0.60, L = \small {}]{v8}
\Vertex[style = unlabeledStyle, x = 0.55, y = 0.60, L = \small {}]{v9}
\Vertex[style = unlabeledStyle, x = 0.65, y = 0.50, L = \small {}]{v10}
\Edge[](v0)(v3)
\Edge[](v1)(v3)
\Edge[](v8)(v3)
\Edge[](v9)(v3)
\Edge[](v10)(v3)
\Edge[](v0)(v2)
\Edge[](v1)(v2)
\Edge[](v4)(v2)
\Edge[](v5)(v2)
\Edge[](v6)(v2)
\Edge[](v7)(v2)
\Edge[](v4)(v1)
\Edge[](v8)(v1)
\Edge[](v4)(v7)
\Edge[](v6)(v7)
\Edge[](v0)(v5)
\Edge[](v6)(v5)
\Edge[](v1)(v0)
\Edge[](v9)(v0)
\Edge[](v9)(v10)
\Edge[](v8)(v10)
\end{tikzpicture}
}}

\caption{The two instances of configuration $H$.}
\label{figH}
\end{figure}
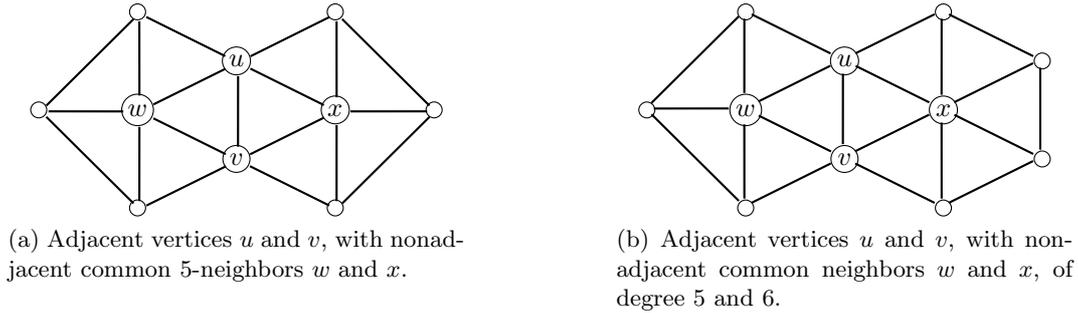
Assume that $G$ has minimum degree 5 and no separating 3-cycle, but also has no
copy of $H$.  This assumption leads to a contradiction, which implies the
result.  An immediate consequence of this assumption (by Pigeonhole) is that
the number of 5-neighbors of each vertex $v$ is at most $\floor{\frac{d(v)}2}$.
 Below, when we verify that each vertex finishes with nonnegative
charge, we consider both the degree of $v$ and its number of 5-neighbors.
We write {\it $(a,b)$-vertex} to denote a vertex of degree $a$ that has $b$
5-neighbors.

We assign to each vertex $v$ a charge $\ch(v)$, where $\ch(v)=d(v)-6$.  
Note that $\sum_{v\in V}\ch(v) = 2|E(G)| - 6|V(G)|$.
Since $G$ is a plane triangulation, Euler's formula implies that 
$2|E(G)| - 6|V(G)| = -12$.  
Now we redistribute the charge, without changing the sum, so that each vertex
finishes with nonnegative charge. 
This redistribution is called \emph{discharging}, and we write $\ch^*(v)$ to
denote the charge at each vertex $v$ after discharging.  Since each vertex
finishes with nonnegative charge, we get the obvious contradiction 
$-12=\sum_{v\in V}\ch(v)=\sum_{v\in V}\ch^*(v) \ge 0$.
 We redistribute the charge via the following three discharging rules, which we
apply simultaneously everywhere they are applicable.
\begin{enumerate}
\item[(R1)] Each $7^+$-vertex gives charge $\frac13$ to each 5-neighbor.
\item[(R2)] Each $7^+$-vertex gives charge $\frac17$ to each 6-neighbor that has at
least one 5-neighbor.
\item[(R3)] Each $6$-vertex gives charge $\frac27$ to each 5-neighbor.
\end{enumerate}
\bigskip

We now verify that after discharging, each vertex $v$ has nonnegative charge.
We repeatedly use that $G$ has no copy of configuration $H$. In particular, this
implies that the number of 5-neighbors for each vertex $v$ is at most
$\frac{d(v)}2$.
\bigskip

${\bf d(v)=5}$:  
Each $(5,0)$-vertex $v$ has five $6^+$-neighbors, 
so $\ch^*(v) \ge -1 + 5\left(\frac27\right) > 0$.
Each $(5,1)$-vertex $v$ has four $6^+$-neighbors, at least two of which are
$7^+$-neighbors; so
$\ch^*(v) \ge -1 + 2\left(\frac13\right) + 2\left(\frac27\right) > 0$. 
Each $(5,2)$-vertex $v$ has three $7^+$-neighbors (otherwise $G$ contains a copy
of $H$), so  
$\ch^*(v) = -1 + 3\left(\frac13\right) = 0$. 

${\bf d(v)=6}$: 
Each $(6,0)$-vertex $v$ has $\ch^*(v) = \ch(v) = 0$.
Each $(6,1)$-vertex $v$ has at least two $7^+$-neighbors, 
so $\ch^*(v) \ge 0 + 2\left(\frac17\right) - \left(\frac27\right) = 0$.
Each $(6,2)$-vertex $v$ has four $7^+$-neighbors, so    
$\ch^*(v) = 0 + 4\left(\frac17\right) - 2\left(\frac27\right) = 0$.

${\bf d(v)=7}$: 
Each $(7,0)$-vertex $v$ has $\ch^*(v) \ge 1 - 7\left(\frac17\right)=0$. 
Each $(7,1)$-vertex $v$ has six $6^+$-neighbors, at least two of which are
$7^+$-vertices (namely, the neighbors that are two further clockwise and two
further counterclockwise around $v$ from the $5$-vertex; otherwise $G$ has a
copy of $H$).  So $\ch^*(v) \ge 1 - 1\left(\frac13\right) -
4\left(\frac17\right)>0$.  Each $(7,2)$-vertex has five $6^+$-neighbors, at
least three of which are $7^+$-vertices; so $\ch^*(v) \ge 1 -
2\left(\frac13\right) - 2\left(\frac17\right)>0$. 
Each $(7,3)$-vertex has four $7^+$-neighbors, so
$\ch^*(v) = 1 - 3\left(\frac13\right)=0$. 

${\bf d(v)= 8}$: Now $v$ has at most four 5-neighbors, and gives each of these charge
$\frac13$; also $v$ gives each other neighbor charge at most $\frac17$.  Thus
$\ch^*(v)\ge 8-6-4(\frac13)-4(\frac17)>0$.

${\bf d(v)\ge 9}$: Now $v$ gives each neighbor charge at most $\frac13$, so
$\ch^*(v)\ge d(v)-6-d(v)(\frac13) = \frac23(d(v)-9)\ge 0$.
\smallskip

Thus $-12 = \sum_{v\in V}\ch(v) = \sum_{v\in V}\ch^*(v) \geq 0$.  This
contradiction implies the result.
\end{proof}

\section{Discharging}
\label{discharging}

In this section we present the discharging argument for the proof of
Theorem~\ref{mainthm2}.  
It is convenient to collect all of the reducibitiy lemmas that we use
to analyze the discharging (but prove later).

\begin{lemma8}
Every independent set $J$ in a minimal $G$ with $|J| = 2$, satisfies $|N(J)| \ge 9$.
\end{lemma8}

\begin{lemma9}
A minimal $G$ cannot have two nonadjacent $5$-vertices with at least two common 
neighbors. In particular, each vertex $v$ in $G$ has $\frac12d(v)$ or more
$6^+$-neighbors.
\end{lemma9}

\begin{lemma17}
Every minimal $G$ has no $6$-vertex $v$ with $6^-$-neighbors $u_1$, $u_2$, and
$u_3$ that are pairwise nonadjacent.
\end{lemma17}

\begin{lemma18}
Every minimal $G$ has no $6$-vertex v with pairwise nonadjacent neighbors $u_1$,
$u_2$, and $u_3$, where $d(u_1) = 5$, $d(u_2) \le 6$, and $d(u_3) = 7$.
\end{lemma18}

\begin{lemma19}
Let $u_1$ be a $6$-vertex with nonadjacent vertices $u_2$ and $u_3$ each at
distance two from $u_1$, where $u_2$ is a $5$-vertex and $u_3$ is a $6^-$-vertex.
A minimal $G$ cannot have $u_1$ and $u_2$ with two common neighbors, and also
$u_1$ and $u_3$ with two common neighbors.
\end{lemma19}

\begin{lemma20}
Every minimal $G$ has no $7$-vertex $v$ with a $5$-neighbor and two other
$6^-$-neighbors, $u_1$, $u_2$, and $u_3$, that are pairwise nonadjacent.
\end{lemma20}

\begin{lemma21}
Let $v_1$, $v_2$, $v_3$ be the corners of a $3$-face, each a $6^+$-vertex. Let
$u_1$, $u_2$, $u_3$ be the other pairwise common neighbors of $v_1$, $v_2$,
$v_3$, i.e., $u_1$ is adjacent to $v_1$ and $v_2$, $u_2$ is adjacent to $v_2$
and $v_3$, and $u_3$ is adjacent to $v_3$ and $v_1$. We cannot have
$|N(\{u_1,u_2,u_3\})| \le 13$. In particular, we cannot have $d(u_1) = d(u_2) =
5$ and $d(u_3) \le 6$.
\end{lemma21}

\begin{lemma22}
Let $u_1$ be a $7$-vertex with nonadjacent $5$-vertices $u_2$ and $u_3$ each
at distance two from $u_1$. A minimal $G$ cannot have $u_1$ and $u_2$ with two
common neighbors and also $u_1$ and $u_3$ with two common neighbors.
\end{lemma22}

\begin{lemma23}
Suppose that a minimal $G$ contains a $7$-vertex $v$ with no $5$-neighbor. Now
$v$ cannot have at least five $6$-neighbors, each of which has a $5$-neighbor.
\end{lemma23}

\begin{mainfirst}
Every planar graph $G$ has independence ratio at least $\frac{3}{13}$.
\end{mainfirst}

\begin{proof}
We assume that the theorem is false, and let $G$ be a minimal counterexample to
the theorem; by ``minimal'' we mean having the fewest vertices and, subject to
that, the fewest non-triangular faces (thus, $G$ is a triangulation).
We will use discharging with initial charge $\ch(v)=d(v)-6$. 
We use the following five discharging rules to guarantee that each
vertex finishes with nonnegative charge, which yields a contradiction.

\begin{enumerate}
\item[(R1)] 
Each 6-vertex gives $\frac12$ to each 5-neighbor unless either they share a common
6-neighbor and no common 5-neighbor or else the 5-neighbor receives charge from
at least four vertices; in either of these cases, the 6-vertex gives the
5-neighbor $\frac14$.

\item[(R2)] 
Each $8^+$-vertex $v$ gives $\frac14 + \frac{h_w}{8}$ to each $6^-$-neighbor
$w$ where $h_w$ is the number of $7^+$-vertices in $N(v) \cap N(w)$.

\item[(R3)] 
Each 7-vertex gives $\frac12$ to each isolated 5-neighbor; gives 0 to each
crowded 5-neighbor; gives $\frac14$ to each other 5-neighbor; and gives
$\frac14$ to each 6-neighbor unless neither the 7-vertex nor the 6-vertex has a
5-neighbor.

\item[(R4)] 
After applying (R1)--(R3), each 5-vertex with positive charge
splits it equally among its 6-neighbors that gave it $\frac12$.

\item[(R5)] 
After applying (R1)--(R4), each 6-vertex with positive charge
splits it equally among its 6-neighbors with negative charge.
\end{enumerate}

Now we show that after applying these five discharging rules, each vertex $v$
finishes with nonnegative charge, i.e., $\ch^*(v)\ge 0$.  (It is worth noting
that if some vertex $v$ has nonnegative charge after applying only (R1)--(R3),
then $v$ also has nonnegative charge after applying (R1)--(R5), i.e.,
$\ch^*(v)\ge 0$.  In fact, the analysis for most cases only needs (R1)--(R3).
The final two rules are used only in Cases (iv)--(vi), near the end of the
proof.) Since the sum of the initial charges is $-12$, this contradicts our
assumption that $G$ was a minimal counterexample.  Subject to proving the
needed reducibility lemmas, this contradiction completes the proof of
Theorem~\ref{mainthm2}.
\bigskip

$\mathbf{d(v)\ge 8}$:
We will show that $v$ gives away charge at most $\frac{d(v)}4$.  To see that it
does, let $v$ first give charge $\frac14$ to each neighbor.  Now let each
$6^-$-neighbor $w$ take $\frac18$ from each $7^+$-vertex in $N(v) \cap N(w)$. 
Since $G[N(v)]$ is a cycle, each $7^+$-neighbor gives away at most the
$\frac14$ it got from $v$.  Each neighbor of $v$ has received at least as much
charge as by rule (R2) and $v$ has given away charge $\frac{d(v)}{4}$. 
Now $\ch^*(v)\ge \ch(v)-\frac14d(v)=d(v)-6-\frac14d(v)=\frac34(d(v)-8)\ge 0$.
\smallskip

$\mathbf{d(v)=7}$:
Let $u_1,\ldots, u_7$ denote the neighbors of $v$ in clockwise order.  First
suppose that $v$ has an isolated 5-neighbor.  By Lemma~\ref{alpha7566}, the
subgraph induced by the remaining $6^-$-neighbors must have independence number
at most 1.  Hence $v$ gives away charge at most either $\frac12+\frac12$ or
$\frac12+2(\frac14)$; in either case, $\ch^*(v)\ge \ch(v)-1 = 0$.  Assume
instead that $v$ has no isolated 5-neighbor.  Suppose first that $v$ has a
(non-isolated) 5-neighbor.  Now $v$ has at most five total $6^-$-neighbors,
again by Lemma~\ref{alpha7566}.  If $v$ has at most four $6^-$ neighbors, then,
since each $6^-$-neighbor receives charge at most $\frac14$, 
we have $\ch^*(v)\ge \ch(v)-4(\frac14)=0$.
By
Lemma~\ref{alpha7566}, if $v$ has exactly five $6^-$-neighbors, then one is a
crowded 5-neighbor, which receives no charge from $v$.  So, again, 
$\ch^*(v)\ge 1-4(\frac14)=0$.
Finally, suppose that $v$ has only $6^+$-neighbors.
By Lemma~\ref{reduce7vertex}, $v$ gives charge to at most four $6$-neighbors, so
$\ch^*(v)\ge \ch(v)-4(\frac14) = 0$.
\smallskip

$\mathbf{d(v)=5}$:
Since $\ch(v)=-1$, we must show that $v$ receives total charge at least 1.
Let $u_1, \ldots, u_5$ be the neighbors of $v$.  First suppose that $v$ has
five $6^+$-neighbors.  Now $v$ will receive charge at least $4(\frac14)$ unless
exactly two of these neighbors are 7-vertices for which $v$ is a crowded 5-neighbor. 
However, in this case the other three neighbors are all $6$-neighbors, so 
$\ch^*(v)\ge -1+2(\frac14)+\frac12=0$.  Now suppose that $v$ has exactly four
$6^+$-neighbors, say $u_1,\ldots,u_4$.  If $v$ receives charge from each, then
$\ch^*(v)\ge -1 + 4(\frac14)=0$; so suppose that $v$ receives charge from at
most three neighbors.  In total, $v$ receives charge at least $\frac12$ from $u_1$
and $u_2$: at least $2(\frac14)$ if $u_1$ is not a 6-vertex and at least $\frac12+0$ if
$u_1$ is a 6-vertex.  Similarly, $v$ receives at least $\frac12$ in total from
$u_3$ and $u_4$; so, $\ch^*(v)\ge -1 +2(\frac12)=0$.  Now suppose that $v$ has
exactly three $6^+$-neighbors, say $u_1, u_2, u_3$.  Lemma~\ref{alpha55} implies that
$u_1$, $u_2$, $u_3$ are consecutive neighbors of $v$.  If $u_1$
and $u_3$ are both 6-vertices, then $v$ receives charge $\frac12$ from
each.  If both are $7^+$-vertices, then $v$ receives charge $\frac14$ from each
and charge $\frac12$ from $u_2$.  So assume that exactly one of $u_1$ and $u_3$
is a 6-vertex, say $u_1$.  Now $v$ receives charge $\frac12$ from $u_1$ and
charge $\frac14$ from each of $u_2$ and $u_3$, for a total of $\frac12+2(\frac14)$.
In every case $\ch^*(v)\ge \ch(v)+1=0$.
\smallskip

$\mathbf{d(v)=6}$:
Note that (R5) will never cause a 6-vertex to have
negative charge.  Thus, in showing that a 6-vertex has nonnegative charge, we
need not consider it.

Clearly, a 6-vertex with no 5-neighbor finishes (R1)--(R3) with nonnegative
charge.  Suppose that $v$ is a 6-vertex with exactly one 5-neighbor.  We will
show that $v$ finishes (R1)--(R3) with charge at least $\frac14$.  Let
$u_1,\ldots, u_6$ denote the neighbors of $v$ and assume that $u_1$ is the only
5-vertex.  By Lemma~\ref{alpha666}, at least one of $u_1, u_3, u_5$ is a
$7^+$-vertex, so it gives $v$ charge $\frac14$.  If one of $u_6$ and $u_2$ is a
6-vertex, then $v$ gives charge only $\frac14$ to $u_1$, finishing with charge at
least $2(\frac14)-\frac14$.  Otherwise, $v$ receives charge at least $\frac14$
from each of $u_6$ and $u_2$, so finishes with charge at least
$3(\frac14)-\frac12$.  Similarly, if $v$ has no 5-neighbor and at least one
$8^+$-neighbor, then $v$ finishes (R1)--(R3) with charge at least $\frac14$.

Now suppose that $v$ has at least two 5-neighbors.
By Lemma~\ref{alpha55}, At most one of $u_1, u_3, u_5$ can be a 5-vertex. 
Similarly, for $u_2, u_4, u_6$; hence, assume that $v$ has exactly two
5-neighbors.  These 5-neighbors can either be ``across'', say $u_1$
and $u_4$, or ``adjacent'', say $u_1$ and $u_2$.

Suppose that $v$ has 5-neighbors $u_1$ and $u_4$.  Note that all of its
remaining neighbors must be $6^+$-vertices. At least one of $u_1, u_3,
u_5$ must be a $7^+$-vertex; similarly for $u_2,u_4,u_6$.  Now we show that the
total net charge that $v$ gives to $u_3, u_4, u_5$ is 0.  Similarly, the total
net charge that $v$ gives to $u_6, u_1, u_2$ is 0.  If both $u_3$ and $u_5$ are
$7^+$-vertices, then $v$ gets $\frac14$ from each and gives $\frac12$ to $u_4$.
Otherwise, one of $u_3$ and $u_5$ is a 6-vertex and the other is a $7^+$-vertex;
now $v$ gets $\frac14$ from the $7^+$-vertex and gives only $\frac14$ to $u_4$.
The same is true for $u_6, u_1, u_2$.  Thus, $v$ finishes with charge 0.

\begin{figure}
\centering
\begin{tikzpicture}[scale = 13]
\tikzstyle{VertexStyle} = []
\tikzstyle{EdgeStyle} = []
\tikzstyle{labeledStyle}=[shape = circle, minimum size = 6pt, inner sep = 1.2pt, draw]
\tikzstyle{unlabeledStyle}=[shape = circle, minimum size = 6pt, inner sep = 1.2pt, draw, fill]
\Vertex[style = labeledStyle, x = 0.550, y = 0.850, L = \small {$u_1$}]{v0}
\Vertex[style = labeledStyle, x = 0.650, y = 0.850, L = \small {$u_2$}]{v1}
\Vertex[style = labeledStyle, x = 0.600, y = 0.750, L = \small {$v$}]{v2}
\Vertex[style = labeledStyle, x = 0.700, y = 0.750, L = \small {$u_3$}]{v3}
\Vertex[style = labeledStyle, x = 0.650, y = 0.650, L = \small {$u_4$}]{v4}
\Vertex[style = labeledStyle, x = 0.550, y = 0.650, L = \small {$u_5$}]{v5}
\Vertex[style = labeledStyle, x = 0.500, y = 0.750, L = \small {$u_6$}]{v6}
\Vertex[style = labeledStyle, x = 0.600, y = 0.950, L = \small {$w_2$}]{v7}
\Vertex[style = labeledStyle, x = 0.450, y = 0.900, L = \small {$w_1$}]{v8}
\Vertex[style = labeledStyle, x = 0.750, y = 0.900, L = \small {$w_3$}]{v9}
\Vertex[style = labeledStyle, x = 0.750, y = 0.650, L = \small {$w_4$}]{v10}
\Vertex[style = labeledStyle, x = 0.600, y = 0.550, L = \small {$w_5$}]{v11}
\Vertex[style = labeledStyle, x = 0.450, y = 0.650, L = \small {$w_6$}]{v12}
\Edge[label = \small {}, labelstyle={auto=right, fill=none}](v0)(v2)
\Edge[label = \small {}, labelstyle={auto=right, fill=none}](v1)(v2)
\Edge[label = \small {}, labelstyle={auto=right, fill=none}](v3)(v2)
\Edge[label = \small {}, labelstyle={auto=right, fill=none}](v4)(v2)
\Edge[label = \small {}, labelstyle={auto=right, fill=none}](v5)(v2)
\Edge[label = \small {}, labelstyle={auto=right, fill=none}](v6)(v2)
\Edge[label = \small {}, labelstyle={auto=right, fill=none}](v1)(v9)
\Edge[label = \small {}, labelstyle={auto=right, fill=none}](v3)(v9)
\Edge[label = \small {}, labelstyle={auto=right, fill=none}](v7)(v9)
\Edge[label = \small {}, labelstyle={auto=right, fill=none}](v0)(v1)
\Edge[label = \small {}, labelstyle={auto=right, fill=none}](v3)(v1)
\Edge[label = \small {}, labelstyle={auto=right, fill=none}](v7)(v1)
\Edge[label = \small {}, labelstyle={auto=right, fill=none}](v0)(v7)
\Edge[label = \small {}, labelstyle={auto=right, fill=none}](v0)(v8)
\Edge[label = \small {}, labelstyle={auto=right, fill=none}](v0)(v6)
\Edge[label = \small {}, labelstyle={auto=right, fill=none}](v6)(v8)
\Edge[label = \small {}, labelstyle={auto=right, fill=none}](v7)(v8)
\Edge[label = \small {}, labelstyle={auto=right, fill=none}](v3)(v10)
\Edge[label = \small {}, labelstyle={auto=right, fill=none}](v4)(v10)
\Edge[label = \small {}, labelstyle={auto=right, fill=none}](v3)(v4)
\Edge[label = \small {}, labelstyle={auto=right, fill=none}](v5)(v4)
\Edge[label = \small {}, labelstyle={auto=right, fill=none}](v4)(v11)
\Edge[label = \small {}, labelstyle={auto=right, fill=none}](v5)(v11)
\Edge[label = \small {}, labelstyle={auto=right, fill=none}](v5)(v6)
\Edge[label = \small {}, labelstyle={auto=right, fill=none}](v5)(v12)
\Edge[label = \small {}, labelstyle={auto=right, fill=none}](v6)(v12)
\end{tikzpicture}
\caption{The closed neighborhood of $v$ and some nearby vertices.}
\label{fig:6vert}
\end{figure}
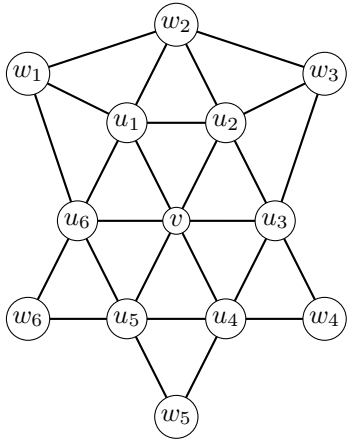
Suppose instead that $v$ has 5-neighbors $u_1$ and $u_2$.
By Lemmas~\ref{alpha666} and \ref{alpha6567} either both of $u_3$ and
$u_5$ are $7^+$-vertices or one is a $6$-vertex and the other an $8^+$-vertex.  
The same holds for $u_4$ and $u_6$.  
Let $w_1,\ldots,w_6$ be the common neighbors of successive pairs of vertices in
the list $u_6, u_1, u_2, u_3, u_4, u_5, u_6$.
Note that $w_1 \adj w_2$, since $u_1$ is a 5-vertex and
$\{v,u_6,w_1,w_2,u_2\}\subseteq N(u_1)$.  Similarly, $w_2\adj w_3$.
(See Figure~\ref{fig:6vert}.)
By Lemma~\ref{alpha55}, since $G$ has no separating 3-cycle, $w_1$ and $w_3$ are
$6^+$-vertices.
Consider the possible degrees for $u_3$, $u_4$, $u_5$, $u_6$.  Up to symmetry,
they are (i) $7^+, 7^+, 7^+, 7^+$, (ii) $7^+, 8^+, 7^+, 6$, (iii) $8^+, 7^+, 6, 7^+$,
(iv) $8^+, 6, 6, 8^+$, (v) $6, 6, 8^+, 8^+$, and (vi) $6, 8^+, 8^+, 6$.

In \textbf{Case (i)}, $v$ receives charge at least $4(\frac14)$, so $\ch^*(v)\ge 0$.
In \textbf{Case (ii)}, $v$ receives charge at least
$\frac14+(\frac14+\frac18+\frac18)+\frac14$, so $\ch^*(v)\ge
1-(\frac12+\frac14)\ge 0$.
In \textbf{Case (iii)}, $v$ receives charge at least
$(\frac14+\frac18)+\frac14+\frac14=\frac78$.  
Recall that $w_3$ is a $6^+$-vertex, by Lemma~\ref{alpha55}.  
If $w_2$ is a $6^+$-vertex, then $v$ gives only $\frac14$ to $u_2$, so
$\ch^*(v)\ge \frac78 - (\frac14+\frac12) = 0$.  So suppose that $w_2$ is a
5-vertex.  Now in each case $v$ gets charge at least $\frac18$ back from $u_2$,
via (R4).  If $w_3$ is a 6-vertex, then $u_2$ receives charge
$2(\frac12)+\frac14$ and sends back $\frac18$ to each of $v$ and $w_3$. 
Otherwise, $w_3$ is a $7^+$-vertex, so $u_3$ sends $u_2$ charge at least
$\frac38$, and $v$ gets back at least $\frac18$.  Thus, in each instance of
Case (iii), we have $\ch^*(v)\ge 0$.  So we are in Cases (iv), (v), or (vi).

\textbf{Case (iv): $8^+,6,6,8^+$.}
If $w_2$ is a $6^+$-vertex, then both $u_1$ and $u_2$ are sent charge by four
vertices and hence $v$ gives away at most $2(\frac14)$. Since $v$ gets at least
$\frac14$ from each of $u_3$ and $u_6$, we have $\ch^*(v)\ge
2(\frac14)-2(\frac14)=0$.  Hence, we assume that $w_2$ is a $5$-vertex.  

Now if $w_1$ is a $6$-vertex, then
$u_1$ receives charge $\frac54$, so gives back $\frac18$ to $v$.  If
instead $w_1$ is a $7^+$-vertex, then $u_1$ receives charge at least $\frac34$
from $v$ and $w_1$ together and then charge at least $\frac14 + \frac18$ from
$u_6$ for a total of $\frac98$. Since $u_1$ has only one $6$-neighbor, it
gives the extra $\frac18$ back to $v$ by (R4).  The same holds for $u_2$, so
$v$ gets $\frac18$ back from each of $u_1$ and $u_2$. 
So, the total charge that $v$ gets from $u_6$, $u_1$, $u_2$, $u_3$ is at least
$\frac14+\frac18+\frac18+\frac14=\frac34$.

Suppose that $u_4$ has at least two 5-neighbors. Now one of them,
call it $x$, is a common neighbor with either $u_3$ or $u_5$, so we can apply
Lemma~\ref{alpha665big} to $\set{v,w_2,x}$ (again $x\nonadj w_2$, since $w_2$
has two other 5-neighbors; $x$ cannot be identified with one of these other
5-neighbors, since $G$ has no separating 3-cycle).  
Similarly, $u_5$ has at most one
5-neighbor.  Hence, by our argument above, both $u_4$ and $u_5$ finish
(R1)--(R3) with charge at least $\frac14$.  Now we show that $u_4$ has at most
three 6-neighbors; similarly for $u_5$.

Suppose that $u_4$ has at least four $6$-neighbors.  
Define $y$ by $N(u_4) = \{v,u_3,w_3,y,w_4,u_5\}$.
Recall that $w_2 \adj w_1$ and $w_2 \adj w_3$, 
as noted before Case (i).
If $y$ is a $6$-vertex, then we can apply
Lemma~\ref{alpha665big} to $\{u_5,y,u_1\}$.  
(We cannot have $y=w_1$, since letting $J=\{u_2,u_5,w_1\}$ gives $|J|=3$
and $|N(J)|\le 6+6+5-1-2-3=11$, which contradicts
Lemma~\ref{CrunchToOneVertexNeighborhoodSize}.)
So instead, both $w_4$ and $w_5$ must be
$6$-vertices.  We can apply Lemma~\ref{alpha665big} to $\{v, w_2, w_5\}$
unless $w_2\adj w_5$, so assume this.  Also, we can apply
Lemma~\ref{alpha665big} to $\{v, w_2, w_4\}$ unless $w_2\adj w_4$; so assume
this.  
Hence, $N(w_2) \supseteq \{u_1, u_2, w_1, w_3, w_4, w_5\}$, which is a contradiction 
since $d(w_2)=5$.

Thus, we conclude that $u_4$ has
at most two 6-neighbors other than $u_5$, so at most two 6-neighbors that finish
(R1)--(R3) with negative charge.  An analogous argument holds for $u_5$.
%
Hence $v$ gets at least $\frac18$ from each of $u_4$ and $u_5$
via (R5), 
so $\ch^*(v)\ge 0-2(\frac12)+\frac34+2(\frac18)=0$.

\textbf{Case (v): $6,6,8^+,8^+$.}
Note that $v$ receives charge at least $2(\frac38)=\frac34$ from $u_5$ and
$u_6$.  If $w_2$ is a $6^+$-vertex, then $u_1$ receives charge from
four neighbors, so $v$ gives away charge at most $\frac14+\frac12$.  Thus
$\ch^*(v)\ge 0$.  So assume $w_2$ is a 5-vertex.
First, we show that $v$ gets back at least $\frac18$ from $u_1$.  If
$d(w_1)=6$, then $u_1$ gets charge $\frac12+\frac14+\frac12=\frac54$,
so returns charge $\frac18$ to each of $v$ and $w_1$.  Otherwise $w_1$ is a
$7^+$-vertex, so $u_6$ sends charge $\frac38$ to $u_1$, and $u_1$ returns at
least $\frac18$ to $v$.  Thus, the total charge that $v$ gets from $u_5$, $u_6$,
and $u_1$ is at least $2(\frac38)+\frac18=\frac78$.

If $w_3$ is a 6-vertex, then $v$ gets back charge $\frac18$ from
$u_2$, via (R4), so $\ch^*(v)0-2(\frac12)+\frac78+\frac18=0$.  Instead, assume
$w_3$ is a $7^+$-vertex.  Now we show that $v$ gets charge at least $\frac18$
from $u_3$ by (R5).  Let $y$ be the neighbor of $u_3$ other than $v, u_2, w_3,
w_4, u_4$.  Applying Lemma~\ref{alpha6567} to $\{u_2,u_4,y\}$, shows that $y$
is an $8^+$-vertex.  If $w_4$ is a 5-vertex, then we apply
Lemma~\ref{alpha665big} to $\{v,w_2,w_4\}$ to get a
contradiction ($w_4$ cannot be adjacent to $w_2$, since $w_2$ already has two
other 5-neighbors, and $w_4$ cannot be identified with $u_1$ or $w_2$, since
$G$ has no separating 3-cycles).  Hence $w_4$ is a $6^+$-vertex.  So $u_3$
receives charge at least $\frac14$ from $w_3$ and at least $\frac14+\frac18$
from $y$.  After $u_3$ gives charge $\frac14$ to $u_2$, it has charge at least
$\frac38$.  So, by (R5), $u_3$ gives each of its at most three 6-neighbors
(including $v$) charge at least $\frac13(\frac38)=\frac18$.  Thus, $\ch^*(v)\ge
-1+\frac78+\frac18 = 0$.

\textbf{Case (vi): $6,8^+,8^+,6$.}
First suppose that $w_2$ is a $6^+$-vertex.  Note that $v$ gets charge at least
$2(\frac38)$ from $u_4$ and $u_5$, so it suffices to show that $v$ gives net
charge at most $\frac38$ to each of $u_1$ and $u_2$.  We consider $u_1$; the
case for $u_2$ is symmetric.  If $w_1$ gives charge to $u_1$, then $u_1$ receives
charge from four neighbors, so it gets charge only $\frac14$ from $v$.  Recall
that $w_1$ must be a $6^+$-vertex, as noted before Case (i).  Thus $w_1$ fails
to give charge to $u_1$ only if $u_1$ is a crowded 5-neighbor of $w_1$; 
suppose this is the case.  So $w_1$ is a 7-vertex and $w_2$ is a 6-vertex.
Now $u_1$ gets charge $\frac12+\frac14+\frac12=\frac54$,
so $u_1$ returns charge $\frac18$ to each of $w_2$ and $v$, via (R4), as desired.  
By symmetry, $u_2$ also returns $\frac18$ to $v$.  
Thus $\ch^*(v)\ge 0-2(\frac12)+2(\frac38)+2(\frac18)=0$. 
So instead, assume that $w_2$ is a 5-vertex.

Now we show that $u_2$ returns $\frac18$ to $v$ via (R4).  By symmetry the same
is true of $u_1$.  If $w_3$ is a 6-vertex, then $v$ gets back $\frac18$ from
$u_2$, since $u_2$ receives $\frac12+\frac14+\frac12$ and returns
$\frac18$ to each of $w_3$ and $v$.  
So assume, that
$w_3$ is a $7^+$-vertex.  If $w_4\adj w_2$, then we apply Lemma~\ref{JIsTwo} to
$\{w_2, u_3\}$; so $w_4\nonadj w_2$.  If $w_4$ is a $6^-$-vertex, then we
apply Lemma~\ref{alpha665big} to $\{v,w_2,w_4\}$ to get a contradiction (as above,
$w_4$ cannot be identified with $u_1$ or $w_2$, since $G$ has no separating
3-cycle).   Thus, $w_4$ is a $7^+$-vertex.  So $u_3$ has at least three
$7^+$-neighbors and at most two 6-neighbors.  Thus, after $u_3$ gives charge
$\frac14$ to $u_2$, by (R5) it gives charge $\frac12(\frac12)=\frac14$ to $v$. 
So in each case, $u_3$ gives at least $\frac18$ to $v$ via (R5).  Since the
same is true of $u_6$, we have $\ch^*(v)\ge 0-2(\frac12)+2(\frac38)+2(\frac18)=0$.
\end{proof}

\section{Reducibility}
It is quite useful to know that a minimal counterexample has no separating
3-cycle; we prove this in Lemma~\ref{alphaNoTriangles}.  When proving coloring
results, such a lemma is nearly trivial.  However, for independence results, it
requires much more work.  Albertson proved an analogous lemma when showing
that planar graphs have independence ratio at least $\frac29$.  Our proof
generalizes his to the broader context of showing that a minor-closed family of graphs
has independence ratio at least $c$ for some rational $c$.  We will apply this
lemma to planar graphs and will let $c=\frac3{13}$.

\begin{lem}
Let $c > 0$ be rational. Let $\G$ be a minor-closed family of graphs.  If $G$ is a minimal
counterexample to the statement that every $n$-vertex graph in $\G$ has an
independent set of size at least $cn$, then $G$ has no separating 3-cycle.
\label{alphaNoTriangles}
\end{lem}
\begin{proof}
Suppose to the contrary that $G$ has a separating 3-cycle $X$. Let $A_1$ and
$A_2$ be induced subgraphs of $G$ with $V(A_1) \cap V(A_2) = X$ and $A_1 \cup A_2  = G$.

Our plan is to find big independent sets in two smaller graphs in $\G$ (by minimality)
and piece those independent sets together to get an independent set in $G$ of
size at least $c|G|$ (for brevity, we write $|G|$ for $|V(G)|$).  More
precisely, we consider independent sets in each
$A_i$, either with $X$ deleted, or with some pair of vertices in $X$
identified.  In Claims 1--3, we prove lower bounds on $\alpha(G)$ in terms of
$|A_1|$ and $|A_2|$.  In Claim 4, we examine $|A_1|$ and $|A_2|$ modulo $b$,
where $c=\frac{a}b$ in lowest terms.  In each case, we show that one of the
independent sets constructed in Claims 1--3 has size at least $c|G|$.
Our proof relies heavily on the fact that $\alpha(H)$ is an integer (for every
graph $H$), which often allows us to gain slightly over $c|H|$.

\noindent\textbf{Claim 1. }\textit{$\alpha(G) \ge \ceil{c(|A_1| - 3)} +
\ceil{c(|A_2| - 3)}$.}

The union of the independent sets obtained by applying minimality of $G$ to
$A_1 \setminus X$ and $A_2 \setminus X$ is independent in $G$.

\noindent\textbf{Claim 2. }\textit{$\alpha(G) \ge \ceil{c(|A_i| - 2)} +
\ceil{c|A_j|} - 1$ whenever $\set{i,j} = \set{1,2}$.}

For concreteness, let $i=1$ and $j=2$; the other case is analogous.  Apply
minimality to $A_2$ to get an independent set $I_2$ in $A_2$ with $|I_2| \ge
\ceil{c|A_2|}$.  Form $A_1'$ from $A_1$ by contracting $X$ to a single vertex
$u$.  Apply minimality to $A_1'$ to get an independent set $I_1$ in $A_1'$ with
$|I_1| \ge \ceil{c(|A_1| - 2)}$.  If $u \in I_1$, then $I_1 \cup I_2 \setminus
\set{u}$ is independent in $G$ and has the desired size.  Otherwise, $I_1 \cup
I_2\setminus X$ is an independent set of the desired size in $G$.

\noindent\textbf{Claim 3. }\textit{$\alpha(G) \ge \ceil{c(|A_1| - 1)} + \ceil{c(|A_2| - 1)} - 1$.}

Let $X = \set{x_1,x_2,x_3}$.  For each $k \in \set{1,2}$ and $t \in \set{2,3}$,
form $A_{k,t}$ from $A_k$ by contracting $x_1x_t$ to a vertex $x_{k,t}$.
Applying minimality to $A_{k,t}$ gives an independent set $I_{k,t}$ in
$A_{k,t}$ with $\card{I_{k,t}} \ge \ceil{c(|A_k| - 1)}$.

If at most one of $I_{1,t}$ and $I_{2,t}$ contains a vertex of $X$ (or a
contraction of two vertices in $X$), then to
get a big independent set, we take their union, discarding this at most one
vertex.  Formally,
if $\set{x_{k,t}, x_{5-t}} \cap I_{k,t} = \emptyset$, then $(I_{1,t} \cup
I_{2,t}) \setminus X$ is an independent set in $G$ of the desired size.  So assume
that each of $I_{1,t}$ and $I_{2,t}$ contains a vertex (or a contraction of an
edge) of $X$.

Now we look for a vertex $x_{\ell}$ of $X$ such that each of $I_{1,t}$ and $I_{2,t}$
contains $x_{\ell}$ or a contraction of $x_{\ell}$.
Formally,
if $x_{5-t} \in I_{1,t} \cap I_{2,t}$, then $(I_{1,t} \cup I_{2,t}) \setminus
X$ 
is an independent set in $G$ of the desired size.
Similarly, if $x_{1,t} \in I_{1,t}$ and $x_{2,t} \in I_{2,t}$, then
$(I_{1,t} \cup I_{2,t} \cup \set{x_1}) \setminus \set{x_{1,t}, x_{2,t}}$ is an
independent set in $G$ of the desired size.  

So, by symmetry, we may assume that $x_{1,2} \in I_{1,2}$ and $x_3 \in
I_{2,2}$.  Also, either $x_{1,3} \in I_{1,3}$ or $x_{2,3} \in I_{2,3}$.  If
$x_{1,3} \in I_{1,3}$, then $(I_{2,2} \cup I_{1,3})\setminus \set{x_{1,3}}$ is an
independent set in $G$ of the desired size.  Otherwise, $x_{2,3} \in I_{2,3}$
and $(I_{1,2} \cup I_{2,3} \cup \set{x_1})\setminus \set{x_{1,2}, x_{2,3}}$ is an
independent set in $G$ of the desired size.

\noindent\textbf{Claim 4. }\textit{The lemma holds.}

Let $a$ and $b$ be positive integers such that
$c = \frac{a}{b}$ and $\gcd(a,b) = 1$.  For each $i \in \set{1,2}$, let
$N_i = |A_i| - 3$ and for each $j \in \set{0,1,2,3}$, choose $k_i^j$ such that
$1 \le k_i^j \le b$ and $k_i^j \equiv a(N_i + j) \pmod{b}$.  In other words,
$\ceil{c(N_i+j)}=\frac{a}{b}(N_i + j)+\frac{b-k_i^j}b$.  Intuitively, if there exist
$i$ and $j$ such that $k_i^j$ is small compared to $b$, then we improve our
lower bound on the independence number (in some smaller graph) by the fact that
the independence number is always an integer.  In the present claim, we show
that if some $k_i^j$ is small, then $G$ has an independent set of the desired
size.  In contrast, if all $k_i^j$ are big, then we get a contradiction.

By symmetry, we may assume that $k_1^0 \le k_2^0$.

\noindent\textbf{Subclaim 4a. }\textit{$k_1^0 + k_2^0 \ge 2b + 1 - 3a$ and
$k_1^1 + k_2^3 \ge b + a + 1$ and $k_1^3 + k_2^1 \ge b + a + 1$ and $k_1^2 + k_2^2 \ge b + a + 1$.}

If any independent set constructed in Claims 1--3 has size at least $c|G|$, then
we are done.  So we assume not; more precisely, we assume that each of these
independent sets has size at most $\frac{a|G|-1}{b}$.  Each of the
four desired bounds follow from simplifying the inequalities in Claims 1--3. 
Note that $|G|=N_1+N_2+3$.

By Claim 1, we have $\alpha(G) \ge \ceil{c(|A_1| - 3)} + \ceil{c(|A_2| - 3)}
= \frac{a}{b}(N_1 + N_2) + \frac{b-k_1^0}{b} + \frac{b-k_2^0}{b} =
\frac{a}{b}|G| + \frac{2b -3a - k_1^0 - k_2^0}{b}$.  
Hence $k_1^0 + k_2^0 \ge 2b + 1 - 3a$.

By Claim 2, we have $\alpha(G) \ge \ceil{c(|A_1| - 2)} + \ceil{c|A_2|} - 1 =
\frac{a}{b}(N_1 + 1 + N_2 + 3) + \frac{b-k_1^1}{b} + \frac{b-k_2^3}{b} - 1 = \frac{a}{b}|G| + \frac{2b + a - k_1^1 - k_2^3}{b} - 1$.  
Hence $k_1^1 + k_2^3 \ge b + a + 1$.  Similarly, $k_1^3 + k_2^1 \ge b + a + 1$.

By Claim 3, we have $\alpha(G) \ge \ceil{c(|A_1| - 1)} + \ceil{c(|A_2| - 1)} - 1 \ge \frac{a}{b}(N_1 + 2 + N_2 + 2) + \frac{b-k_1^2}{b} + \frac{b-k_2^2}{b} - 1 = \frac{a}{b}|G| + \frac{2b + a - k_1^2 - k_2^2}{b} - 1$.  
Hence $k_1^2 + k_2^2 \ge b + a + 1$.

Now to get a contradiction, it suffices to show that $k_i^j\le a$ for some
$i\in\{1,2\}$ and some $j\in\{1,2,3\}$; since $k_i^j\le b$ for all $i$ and $j$,
this will contradict one of the equalities above.  

\noindent\textbf{Subclaim 4b. }\textit{Either $k_2^1\le a$ or $k_2^2\le a$.  In
each case we get a contradiction, so the claim is true, and the lemma holds.}

By Subclaim 4a, we have $k_1^0+k_2^0 \ge 2b + 1 - 3a$.  
By symmetry, we assumed $k_2^0\ge k_1^0$, so we have $k_2^0 \ge \frac{2b
+ 1 - 3a}{2}$.  
Since, $k_2^2 \equiv k_2^0 + 2a \pmod{b}$ and $\frac{2b + 1 - 3a}{2} + 2a > b$,
we have $k_2^2 \le k_2^0 + 2a - b$.  Now we consider two cases, depending on
whether $k_2^0\le b-a$ or $k_2^0\ge b-a+1$.  If $k_2^0\le b-a$, then $k_2^2\le
k_2^0+2a-b \le (b-a)+2a-b = a$, a contradiction.  Suppose instead that
$k_2^0\ge b-a+1$.  Now $k_2^1 \equiv k_2^0+a\pmod{b}$.  Since $k_2^0 \ge b-a+1$,
we see that $k_2^0+a \ge b+1$, so $k_2^1 \le k_2^0+a-b\le a$, a contradiction.
\end{proof}

Now we turn to proving a series of lemmas showing that $G$ cannot have too many
$6^-$-vertices near each other.  Many of these lemmas will rely on applications of
the following result, which  
we think may be of independent interest. 
The idea for the proof is to find big independent sets for two smaller graphs,
and piece them together to get a big independent set in $G$.

For $S \subseteq V(G)$, let the \emph{interior} of $S$ be $\I(S) = \{x \in S
\mid N(x) \subseteq S\}$.  For vertex sets $V_1, V_2\subset V(G)$ we write
$V_1\adj V_2$ if there exists an edge $v_1v_2\in E(G)$ with $v_1\in
V_1$ and $v_2\in V_2$; otherwise, we write $V_i\nonadj V_j$.

\begin{lem}\label{uber-red}
Let $\G$ be a minor-closed family of graphs.  Let $G$ be a minimal
counterexample to the statement that every $n$-vertex graph in $\G$ has an
independent set of size at least $cn$ (for some fixed $c > 0$). Let $S_1,
\ldots, S_t$ be pairwise disjoint subsets of a nonempty set $S \subseteq V(G)$
such that $t < |S|$ and $G[S_i]$ is connected for all $i \in \set{1,\ldots,
t}$.  Now there exists $X \subseteq \set{1,...,t}$ such that $S_i\nonadj S_j$
for all distinct $i,j \in X$ and $\alpha\parens{G\brackets{\I(S) \cup
\bigcup_{i \in X} S_i}} < |X| + \ceil{c(|S| - t)}$.
\end{lem}
\begin{proof}
Suppose to the contrary that $\alpha\parens{G\brackets{\I(S) \cup \bigcup_{i
\in X} S_i}} \ge |X| + \ceil{c(|S| - t)}$ for all $X \subseteq \set{1,...,t}$ such
that $S_i\nonadj S_j$ for all distinct $i,j \in X$.  Create $G'$
from $G$ by contracting $S_i$ to a single vertex $w_i$ for each $i \in
\set{1,\ldots, t}$ and removing the rest of $S$.  (Note that we allow $t=0$.)
Since $t < |S|$, we have
$|G'| < |G|$ and hence minimality of $G$ gives an independent set $I$ in $G'$
with $|I| \ge c|G'| = c(|G| - |S| + t)$.  Let $W = I \cap \set{w_1, \ldots,
w_t}$.  By assumption, we have $\alpha\parens{G\brackets{\I(S) \cup
\bigcup_{w_i \in W} S_i}} \ge |W| + \ceil{c(|S| - t)}$.  If $T$ is a maximum
independent set in $G\brackets{\I(S) \cup \bigcup_{w_i \in W} S_i}$, then $(I
\setminus W) \cup T$ is an independent set in $G$ of size at least $|I| - |W| +
|T| \ge c(|G|-|S|+t)-|W|+(|W|+\ceil{c(|S|-t)})\ge c|G|$,
a contradiction.
\end{proof}

We will often apply Lemma~\ref{uber-red} with $S = J \cup N(J)$ for an
independent set $J$.  In this case, we always have $J \subseteq \I(S)$.
We state this case explicitly in Lemma~\ref{ind-red}

\begin{lem}
Let $\G$ be a minor-closed family of graphs.  Let $G$ be a minimal
counterexample to the statement that every $n$-vertex graph in $\G$ has an
independent set of size at least $cn$ (for some fixed $c > 0$).
No independent set $J$ of $G$ and nonnegative integer $k$ simultaneously
satisfy the following conditions.
\label{ind-red}
\begin{enumerate}
\item $|J| \ge c(|N(J)|+k)$.
\item
For at least $|J|-k$ vertices $x \in J$, there is an independent set $\{u_x,v_x\}$ of size 2
in $\displaystyle{N(x) \setminus \bigcup_{y \in J \setminus\set{x}}  N(y)}$.
\end{enumerate}
\end{lem}
\begin{proof}
Suppose the lemma is false.
Let $S = J \cup N(J)$ and $t = |J| - k$. Pick $x_1, \ldots, x_t \in J$ 
satisfying condition (2).  For $i \in \set{1,\ldots, t}$, let $S_i =
\set{x_i, u_{x_i}, v_{x_i}}$.  Applying Lemma~\ref{uber-red}, we get $X
\subseteq \set{1,\ldots,t}$ such that $S_i\nonadj S_j$ for all
distinct $i,j \in X$ and $\alpha\parens{G\brackets{J \cup \bigcup_{i \in X}
S_i}} < |X| + \ceil{c(|S| - t)}$.  By (2), we have $\alpha\parens{G\brackets{J
\cup \bigcup_{i \in X} S_i}} \ge |(J\setminus X)\cup\bigcup_{x\in
X}\set{u_x,v_x}| \ge (|J| - |X|) + 2|X|  = |X| + |J|$.  Hence $|X| +
\ceil{c(|S| - t)} > |X| + |J|$, giving $\ceil{c(|S| - t)} > |J| \ge
\ceil{c(|N(J)|+k)}$ by (1).  But $|S| - t = (|J| + |N(J)|) - (|J| - k) = |N(J)| +
k$; so $\ceil{c(|S|-t)} = \ceil{c(|N(J)|+k)}$, contradicting the previous
inequality.  This contradiction finishes the proof.
\end{proof}

As a simple example of how to apply Lemma~\ref{ind-red}, we note that it immediately
implies that
every planar graph $G$ has independence ratio at least $\frac15$.
By Euler's theorem, $G$ has a $5^-$-vertex $v$.  If $d(v)\le 4$, then let $G' =
G\setminus(v\cup N(v))$.  Let $I'$ be an independent set in $G'$ of size at
least $(n-5)/5$, and let $I = I'\cup\{v\}$.  If instead $d(v)=5$, then apply
Lemma~\ref{ind-red}, with $c=\frac15$, $J=\{v\}$, and $k=0$; since $K_6$ is
nonplanar, $v$ has some pair of nonadjacent neighbors.  This completes the
proof.

\begin{lem}\label{GeneralCrunchToOneVertexNeighborhoodSize}
Let $\G$ be a minor-closed family of graphs.  Let $G$ be a minimal
counterexample to the statement that every $n$-vertex graph in $\G$ has an
independent set of size at least $cn$ (for some fixed $c > 0$).  For any
non-maximal independent set $J$ in $G$, we have
\[\card{N(J)} \ge \floor{\frac{1-c}{c}|J|} + 2.\]
\end{lem}
\begin{proof}
Assume the lemma is false and choose a counterexample $J$ minimizing $|J|$. Suppose $G[J \cup
N(J)]$ is not connected. Now we choose a partition $\set{J_1, \ldots,
J_k}$ of $J$, minimizing $k$, such that $k \ge 2$ and $G[J_i \cup N(J_i)]$ is
connected for each $i \in \set{1,\ldots,k}$.  Applying the minimality of $|J|$ to each $J_i$
we conclude that $|N(J_i)| \ge \floor{\frac{1-c}{c}|J_i|} + 2$ for each $i \in
\set{1,\ldots,k}$.  The minimality of $k$ gives 
$|N(J)|= \card{\bigcup_{i=1}^k N(J_i)} = \sum_{i=1}^k|N(J_i)|$, 
so $|N(J)| \ge 2k +
\sum_{i=1}^k
\floor{\frac{1-c}{c}|J_i|} \ge k + \sum_{i=1}^k \frac{1-c}{c}|J_i| \ge
2+\frac{1-c}{c}|J|$, a contradiction.  Hence, $G[J \cup N(J)]$ is connected.

Let $S = J \cup N(J)$. Apply Lemma~\ref{uber-red} with $t=1$ and $S_1 = S$.
This shows that either $|J| \le \alpha(G[\I(S)]) < \ceil{c(|S| - 1)}$ or
$\alpha(G[S]) < 1 + \ceil{c(|S| - 1)}$, since the only possibilities are
$X=\emptyset$ and $X=\set{1}$.
By assumption $J$ is a counterexample, so $\card{N(J)} \le
\floor{\frac{1-c}{c}|J|} + 1$, which implies that $|S| = |J| + |N(J)| \le |J| +
\floor{\frac{1-c}{c}|J|} + 1 = \floor{\frac{|J|}{c}} + 1$.  
Now $\ceil{c(|S|-1)} \le \ceil{c((\floor{\frac{|J|}{c}}+1)-1)} =
\ceil{c\floor{\frac{|J|}c}}\le \ceil{|J|}=|J|$.  Hence, we cannot have
$X=\emptyset$ in Lemma~\ref{uber-red}.  

Instead, we must have $X=\set{1}$,
which implies that $\alpha(G[S]) < 1 + \ceil{c(|S| - 1)}$. Since $J$ is
non-maximal, we have $S \ne V(G)$, so we may apply minimality of $G$ to
$G[S]$ to conclude that $\alpha(G[S]) \ge \ceil{c|S|}$.  Combining this
inequality with the previous one, we have $\ceil{c|S|} = \ceil{c(|S| - 1)}$.   
Now the upper bound on $\ceil{c(|S| - 1)}$ from the previous paragraph gives
$\ceil{c|S|} = \ceil{c(|S| - 1)}  
\le |J|$. 
Finally, applying Lemma~\ref{uber-red} with $t=0$ (simply deleting $J\cup
N(J)$) shows that $|J| < \ceil{c(|S|)}$.  These two final inequalities
contradict each other, which
finishes the proof.
\end{proof}

Lemmas~\ref{alphaNoTriangles}--\ref{GeneralCrunchToOneVertexNeighborhoodSize}
hold in a more general setting than just $c=\frac3{13}$, as we showed.  In the
rest of this section, we consider only a planar graph $G$ that is minimal among
those with independence ratio less than $\frac3{13}$.  To remind the reader of
this, we often call it a \emph{minimal} $G$.
Applying Lemma~\ref{GeneralCrunchToOneVertexNeighborhoodSize}
with $c = \frac{3}{13}$ gives the following corollary.

\begin{lem}\label{CrunchToOneVertexNeighborhoodSize}
For any non-maximal independent set $J$ in a minimal $G$, we have
\[\card{N(J)} \ge \floor{\frac{10}{3}|J|} + 2.\]
In particular, if $|J|=1$, then $|N(J)|\ge 5$; if $|J|=2$, then $|N(J)|\ge 8$;
and if $|J|=3$, then $|N(J)|\ge 12$.
\end{lem}

The case $|J|=1$ shows that $G$ has minimum degree
$5$, and this is the best we
can hope for when $|J|=1$.  
Recall that $G$ is a planar triangulation, since we chose it to have as few
non-triangular faces as possible.
As a result, we
can improve the bound when $|J|=2$ to $\card{N(J)} \ge 9$.  Similarly, in many
cases we can improve the bound when $|J|=3$ to $\card{N(J)} \ge 13$.  These 
improvements are the focus of the next ten lemmas.  In many instances, the proofs
are easy applications of Lemma~\ref{uber-red}. First, we need a few basic facts
about planar graphs.

\begin{lem}
\label{BasicPlanarFacts}
If $G$ is a plane triangulation with no separating 3-cycle and $\delta(G) = 5$, then
\begin{enumerate}
\item[(a)] If $v \in V(G)$, then $G[N(v)]$ is a cycle; and
\item[(b)] $G$ is $4$-connected with $|V(G)| \ge 12$; and
\item[(c)] If $v,w \in V(G)$ are distinct, then $G[N(v) \cap N(w)]$ is the
disjoint union of copies of $K_1$ and $K_2$.
\end{enumerate}
\end{lem}
\begin{proof}
Plane triangulations are well-known to be $3$-connected.  Property (a) follows
by noting that $G\setminus\set{v}$ is $2$-connected and hence each face boundary is a cycle;
so $G[N(v)]$ has a hamiltonian cycle.  This cycle must be induced since $G$ has
no separating 3-cycle.

For (b), suppose that $G$ has a separating set $\set{x,y,z}$.  Since $G$ has no
separating 3-cycle, we assume that $xy \not \in E(G)$.  By (a), $N(x)$ induces a cycle $C$.
Since $G$ is $3$-connected, $x$ must have a neighbor in each component of $G\setminus\set{x,y,z}$.  
So $C$ has a vertex in each component of $G\setminus\set{x,y,z}$ and hence $C\setminus\set{x,y,z}$ is disconnected.
But $x \not \in V(C)$ and since $xy \not \in E(G)$, also $y \not \in V(C)$.  So, $C\setminus\set{z}$ is disconnected, which is impossible.
Since $G$ is a plane triangulation and 
$\delta(G)=5$, we have $5|G| \le 2|E(G)| = 6|G| - 12$, so $|G| \ge 12$.

By (a) and $\delta(G) = 5$, it follows that no neighborhood contains $K_3$ or
$C_4$. If $G[N(v) \cap N(w)]$ had an induced $P_3$ (path on 3 vertices), then
the neighborhood of the center of this $P_3$ would contain $K_3$ or $C_4$. 
This proves (c).
\end{proof}

\begin{lem}\label{JIsTwo}
Every independent set $J$ in a minimal $G$ with $|J| = 2$, satisfies $\card{N(J)} \ge 9$.
\end{lem}
\begin{proof}
By Lemma~\ref{BasicPlanarFacts}(b), $|G|\ge 12$; so $J$ cannot be a maximal
independent set when $\card{N(J)} \le 7$.  Hence, by Lemma
\ref{CrunchToOneVertexNeighborhoodSize}, we may assume $\card{N(J)} = 8$. 
Let $J = \set{x,y}$.  If we can apply Lemma~\ref{ind-red} with $k=0$, then we
are done.  If we cannot, then by symmetry we may assume that there is no
independent $2$-set in $N(x) \setminus N(y)$.  So $N(x)\setminus N(y)$ is a
clique.  Since $d(x)\ge 5$ and $N(x)$ induces a cycle, $|N(x)\setminus N(y)|\le
2$.  Now, since $x$ is a $5^+$-vertex, $G[N(x) \cap N(y)]$ induces $P_3$; this
contradicts Lemma~\ref{BasicPlanarFacts}(c).
\end{proof}

A direct consequence of Lemma~\ref{JIsTwo} is the following useful fact.
\begin{lem}
A minimal $G$ cannot have two nonadjacent 5-vertices with at least two
common neighbors.  In particular, each vertex $v$ in $G$ has $\frac{d(v)}{2}$
or more $6^+$-neighbors.
\label{alpha55}
\end{lem}
\begin{proof}
The first statement follows immediately from Lemma~\ref{JIsTwo}.  Now we
consider the second.  Let $v$ be a vertex with $d(v)=k$ and neighbors
$u_1,\ldots,u_k$ in clockwise order.  If more than $k/2$ neighbors of $v$ are
5-vertices, then (by Pigeonhole) there exists an integer $i$ such that $u_i$ and
$u_{i+2}$ are 5-vertices (subscripts are modulo $k$).  Now we apply
Lemma~\ref{JIsTwo} to $u_i$ and $u_{i+2}$.  Recall that $u_i$ and $u_{i+2}$ are
nonadjacent, since $G$ has no separating 3-cycle, as shown in
Lemma~\ref{alphaNoTriangles}.
\end{proof}

Now we consider the case when $|J| = 3$.  Lemma
\ref{CrunchToOneVertexNeighborhoodSize} gives $|N(J)| \ge 12$.  Our next few
lemmas show certain conditions under which we can conclude that $|N(J)|\ge 13$.

\begin{lem}\label{ThreeHelper}
Let $J$ be an independent set in a minimal $G$ with $|J| = 3$ and $|N(J)| \ge
12$.  Choose $S_1, S_2 \subseteq J \cup N(J)$ such that $S_1 \cap S_2 =
\emptyset$ and both $G[S_1]$ and $G[S_2]$ are connected. If $\alpha(G[S_i \cup
J]) \ge 4$ for each $i \in \set{1,2}$, then $|N(J)|\ge 13$.
\end{lem}
\begin{proof}
Suppose not and choose a counterexample minimizing $|J \cup N(J)| - |S_1 \cup
S_2|$.  Clearly $|N(J)|=12$.
First we show that $S_1\cup S_2=J\cup N(J)$. It suffices to show
that $G[J\cup N(J)]$ is connected, since then we can add to either $S_1$ or
$S_2$ any vertex in $N(S_1\cup S_2)\setminus (S_1\cup S_2)$.  In particular, we show
that every $x\in J$ satisfies $(x\cup N(x))\cap(\cup_{y\in (J\setminus
\set{x})}(y\cup N(y)))\ne \emptyset$.  Suppose not.  By Lemma~\ref{JIsTwo}, we
have $|\cup_{y\in J\setminus\{x\}}N(y)|\ge 9$.  Now $|\cup_{y\in
J}N(y)|\ge 9+d(x)\ge 14$, a contradiction.
Now we must have $G[J\cup N(J)]$ connected, so we can assume $S_1 \cup S_2 = J
\cup N(J)$.  Similarly, we assume $S_1\adj S_2$.

Now we apply Lemma~\ref{uber-red} with $S=J\cup N(J)$, $t=2$, and $S_1$ and
$S_2$ as above.  Since $S_1\adj S_2$, we have $|X|\le 1$.  We cannot have
$|X|=1$ since, by hypothesis, $\alpha(G[S_i\cup J])\ge 4$ for each
$i\in\set{1,2}$. So suppose that $X=\emptyset$.  Now we have $ \alpha(G[J]) \ge
|J|=3 = \ceil{\frac3{13}(|J\cup N(J)|-2)} = \ceil{\frac3{13}(3+12-2)}$.  This
contradiction completes the proof.
\end{proof}

\begin{lem}\label{JIsThree}
Let $J = \set{u_1, u_2, u_3}$.  If $J$ is an independent set in a minimal $G$ where 
\begin{enumerate}
\item $N(u_1) \setminus (N(u_2) \cup N(u_3))$ contains an independent 2-set; and 
\item $\alpha(G[J \cup N(u_2) \cup N(u_3)]) \ge 4$,
\end{enumerate}then $\card{N(J)} \ge 13$.
\end{lem}
\begin{proof}
Since $G$ is a planar triangulation with minimum degree $5$ and at least three
$6^+$-vertices by Lemma~\ref{alpha55}, we have $5|G| + 3 \le 2|E(G)| = 6|G| -
12$ and hence $|G| \ge 15$.  Thus $J$ cannot be a maximal independent set when
$\card{N(J)} \le 11$. So, by Lemma~\ref{CrunchToOneVertexNeighborhoodSize}, we
know that $\card{N(J)} \ge 12$.  Let $I$ be an independent set of size $2$ in
$N(u_1) \setminus (N(u_2) \cup N(u_3))$.

First, suppose $N(u_2) \cap N(u_3) \ne \emptyset$. We apply Lemma
\ref{ThreeHelper} with $S_1 = \set{u_1} \cup I$ and $S_2 = \set{u_2, u_3} \cup
N(u_2) \cup N(u_3)$.  Clearly, $G[S_1]$ is connected.  Also, $G[S_2]$ is
connected since $N(u_2) \cap N(u_3) \ne \emptyset$, by assumption.  The set $I
\cup \set{u_2, u_3}$ shows that $\alpha(G[S_1 \cup J]) \ge 4$ and hypothesis
(2) shows that $\alpha(G[S_2 \cup J]) \ge 4$.  So the hypotheses of Lemma
\ref{ThreeHelper} are satisfied, giving $\card{N(J)} \ge 13$.

Instead, suppose $N(u_2) \cap N(u_3) = \emptyset$.  
This implies $N(u_2)\setminus(N(u_1)\cup N(u_3))=N(u_2)\setminus N(u_1)$.
If $N(u_2) \setminus N(u_1)$ contains an independent 2-set as well, then
applying Lemma~\ref{ind-red} with $k=1$ gives $\card{N(J)} \ge 13$, as desired. 
Otherwise, $|N(u_2) \setminus N(u_1)| \le 2$, so $G[N(u_2) \cap N(u_1)]$
contains $P_3$, contradicting Lemma~\ref{BasicPlanarFacts}(c).
\end{proof}

One particular case of Lemma~\ref{JIsThree} is easy to verify in our
applications, so we state it separately, as Lemma~\ref{JIsThreeWithSeven}. 
First, we need the following lemma.

\begin{lem}\label{BigAlphaWithSevenPlus}
Let $v$ be a $7^+$-vertex in $G$.  If $S \subseteq V(G)$ with $\set{v} \cup
N(v) \subseteq S$ and $|S| \ge 10$, then $\alpha(G[S]) \ge 4$.
\end{lem}
\begin{proof}
If $d(v) \ge 8$, then the neighbors of $v$ induce an $8^+$-cycle (by
Lemma~\ref{BasicPlanarFacts}(a)), which has independence number at least 4; so
we are done.  So suppose $d(v) = 7$.  Let $u_1,\ldots,u_7$ denote the neighbors
of $v$ in clockwise order; note that $G[N(v)]$ is a 7-cycle, again by
Lemma~\ref{BasicPlanarFacts}(a). Pick $w_1, w_2 \in S \setminus \parens{\set{v}
\cup N(v)}$.  Let $H_i = G[N(v) \setminus N(w_i)]$ for each $i \in \set{1,2}$.
If $H_i$ contains an independent $3$-set $J$ for some $i \in \set{1,2}$, then
$J\cup\{w_i\}$ is the desired independent 4-set, so we are done.  Therefore, we
must have $|H_i| \le 4$ for each $i \in \set{1,2}$.  So, $\card{N(v) \cap
N(w_i)} \ge 3$ and hence Lemma~\ref{BasicPlanarFacts}(c) shows that $N(v)\cap
N(w_i)$ has at least two components; therefore, so does $H_i$. 
It must have exactly two components or we get an independent $3$-set in $H_i$.
Similarly, if $|H_i| = 4$, then $H_i$ has no isolated vertex.  So, either $H_i$
is $2K_2$ or $|H_i| \le 3$. 
Now in each case we get a subdivision of $K_{3,3}$; the branch vertices of one
part are $v,w_1,w_2$ and the branch vertices of the other are three of the
$u_i$.  This contradiction finishes the proof.
\end{proof}

\begin{lem}\label{JIsThreeWithSeven}
Let $J = \set{u_1, u_2, u_3}$. If $J$ is an independent set in a minimal $G$ where 
\begin{enumerate}
\item $N(u_1) \setminus (N(u_2) \cup N(u_3))$ contains an independent 2-set; and 
\item $G[J \cup N(u_2) \cup N(u_3)]$ contains a $7^+$-vertex and its neighborhood,
\end{enumerate}then $\card{N(J)} \ge 13$.
\end{lem}
\begin{proof}
We apply Lemma~\ref{JIsThree} using Lemma~\ref{BigAlphaWithSevenPlus} to verify
hypothesis (2).  To do so, we let $S=\{u_1,u_2,u_3\}\cup N(u_2)\cup N(u_3)$, and
we need that $|\{u_1,u_2,u_3\}\cup N(u_2) \cup N(u_3)| \ge 10$.
This is immediate from Lemma~\ref{JIsTwo}, since $|S|\ge |\{u_1, u_2,
u_3\}|+|N(u_2)\cup N(u_3)|\ge 3+9=12$.
\end{proof}

\begin{lem}\label{ThreeHelperFourteen}
Let $J$ be an independent 3-set in $G$. Choose $S_1, S_2, S_3
\subseteq J \cup N(J)$ such that 
$G[S_i]$ is connected and
$S_i \cap S_j = \emptyset$ 
for all distinct $i,j \in \set{1,2,3}$. If
$|N(J)| \le 13$, then either
\begin{enumerate}
\item $S_i\nonadj S_j$ for some $\set{i,j} \subseteq \set{1,2,3}$; or
\item $\alpha(G[S_i \cup J]) \le 3$ for some $i \in \set{1,2,3}$.
\end{enumerate}
\end{lem}
\begin{proof}
This is an immediate corollary of Lemma~\ref{uber-red} with $S=J\cup N(J)$ and
$t=3$.  If
$S_i\adj S_j$ for all $\set{i,j}\in\set{1,2,3}$, then in Lemma~\ref{uber-red} either
$|X|=1$ or $|X|=0$.  We cannot have $|X|=0$, since $\alpha(G[\I(S)])\ge
\alpha(G[J])\ge |J|=3 =\ceil{\frac3{13}(13+3-3)}$.  Hence $|X|=1$, which
implies (2).
\end{proof}

The next lemma can be viewed as a variant on the result we get by applying
Lemma~\ref{ind-red} with $|J|=3$ and $k=0$ (and $c=\frac3{13}$).  As in that
case, we require that each of $N(u_1)\setminus(N(u_2)\cup N(u_3))$ and
$N(u_2)\setminus(N(u_1)\cup N(u_3))$ contains an independent 2-set.
However, here we do not require that $N(u_3)\setminus (N(u_1)\cup N(u_2))$
contains an independent 2-set.  Instead, we have hypothesis
(2) below.  Not surprisingly, the proof 
is similar to that of Lemma~\ref{ind-red}.

\begin{lem}\label{JIsThreeFourteen}
Let 
$J = \set{u_1, u_2, u_3}$. If $J$ is an independent set in a minimal $G$ such
that 
\begin{enumerate}
\item $N(u_i) \setminus (N(u_j) \cup N(u_3))$ contains an independent 2-set $M_i$ for all $\set{i,j} = \set{1,2}$; and
\item $\alpha(G[J \cup V(H)]) \ge 4$, where $H$ is $u_3$'s component in
$G[\set{u_3} \cup N(J)] \setminus (M_1 \cup
M_2)$, 
\end{enumerate}then $\card{N(J)} \ge 14$.
\end{lem}
\begin{proof}
First, we show that $u_3$ is distance two from each of $u_1$ and $u_2$.  Suppose
not; by symmetry, assume that $u_3$ is distance at least three from $u_1$.  Now
$N(u_3)\setminus(N(u_1)\cup N(u_2)) = N(u_3)\setminus N(u_2)$.  By
Lemma~\ref{BasicPlanarFacts}, $N(u_3)\cap N(u_2)$ consists of disjoint copies
of $K_1$ and $K_2$.  Thus, since $d(u_3)\ge 5$, we see that $N(u_3)\setminus
(N(u_1)\cup N(u_2))$ contains an independent 2-set.  Now, if $\card{N(J)} \le
13$, then applying Lemma~\ref{ind-red} with $k=0$ gives a contradiction. 
Hence, $u_3$ is distance two from each of $u_1$ and $u_2$.  

Choose disjoint subsets $S_1, S_2, S_3 \subset J \cup N(J)$ where $G[S_i]$ is
connected for all $i \in \set{1,2,3}$ and $\set{u_i} \cup M_i \subseteq S_i$
for each $i \in \set{1,2}$ and $u_3 \in S_3$, first maximizing $|S_3|$ and subject to
that maximizing $|S_1| + |S_2| + |S_3|$.  Since $J \subseteq S_1 \cup S_2 \cup
S_3$, maximality of $|S_1| + |S_2| + |S_3|$ gives $S_1 \cup S_2 \cup S_3 = J
\cup N(J)$.  

Now we apply Lemma~\ref{uber-red}, with $S=S_1\cup S_2\cup S_3$.
To get a contradiction, we need only verify, for each possible $X$, that $\alpha(G[\I(S)\cup
\bigcup_{i\in X}S_i]) \ge |X|+\ceil{\frac3{13}(|S|-|J|)}=|X|+3$.  Since $S_3\adj
S_1$ and $S_3\adj S_2$, either $|X|\le 1$ or else $X=\set{1,2}$.
In the latter case, $M_1 \cup M_2 \cup \set{u_3}$ is the desired independent
5-set.  If instead $X=\emptyset$, then $J$ is the desired independent 3-set.

So we must have $X=\set{i}$ for some $i \in \set{1,2,3}$.  If $i \in
\set{1,2}$, then $M_i \cup \set{u_3, u_{3-i}})$ is the desired independent set.
So instead assume that $X=\set{3}$.  But, by the maximality of $|S_3|$, $G[J
\cup S_3]$ contains $u_3$'s component in $G[\set{u_3} \cup N(J)] \setminus M_1
\setminus M_2$.  So by (2), $G[J \cup S_3]$ has an independent 4-set, as desired.
\end{proof}

Again, one particular case of Lemma~\ref{JIsThreeFourteen} is easy to verify, so we state it separately.

\begin{lem}\label{JIsThreeFourteenWithSeven}
Let $J = \set{u_1, u_2, u_3}$.  If $J$ is an independent set in a minimal $G$
such that  
\begin{enumerate}
\item $N(u_i) \setminus (N(u_j) \cup N(u_3))$ contains an independent 2-set $M_i$
for all $\set{i,j} = \set{1,2}$; and 
\item $u_3$'s component $H$ in $G[\set{u_3} \cup N(J)] \setminus (M_1
\cup M_2)$ satisfies $|J \cup V(H)| \ge 10$ and $G[J \cup V(H)]$ contains a
$7^+$ vertex and its neighborhood,
\end{enumerate}then $\card{N(J)} \ge 14$.
\end{lem}
\begin{proof}
We apply Lemma~\ref{JIsThreeFourteen}, using Lemma~\ref{BigAlphaWithSevenPlus} to verify hypothesis (2).  
\end{proof}

Thus far, our lemmas have not focused much on the actual planar embedding of
$G$.  At this point we transition and start analyzing the embedding, as well.

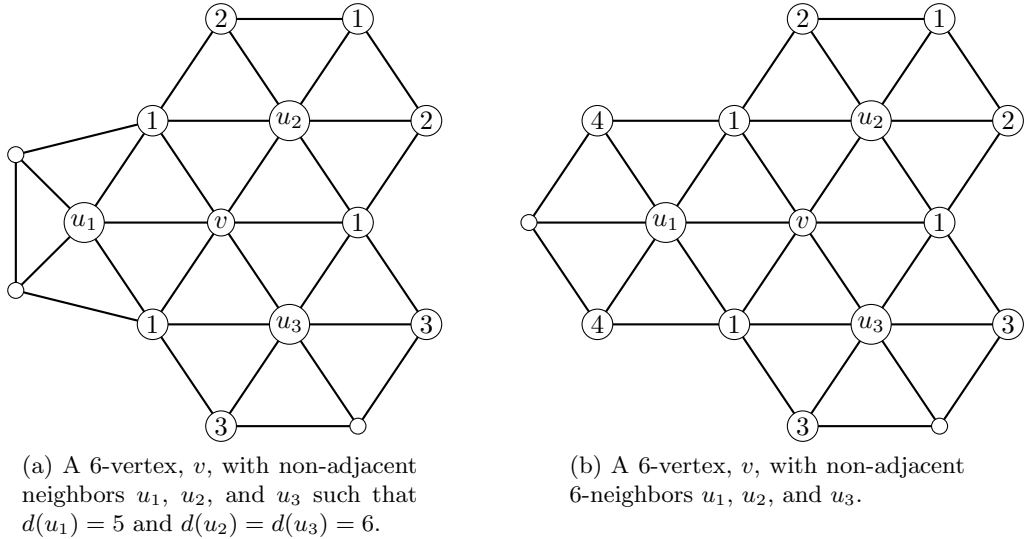
\begin{figure}
\begin{center}
\subfloat[A 6-vertex, $v$, with non-adjacent neighbors $u_1$, $u_2$, and $u_3$ such that
$d(u_1)=5$ and $d(u_2)=d(u_3)=6$.]{\makebox[.45\textwidth]{
\begin{tikzpicture}[xscale = -9, yscale=9]
\tikzstyle{VertexStyle}=[shape = circle, minimum size = 6pt, inner sep = 1.2pt, draw]
\Vertex[x = 0.55, y = 0.55, L = \small {$v$}]{v0}
\Vertex[x = 0.45, y = 0.70, L = \small {$u_2$}]{v1}
\Vertex[x = 0.35, y = 0.55, L = \small {$1$}]{v2}
\Vertex[x = 0.65, y = 0.70, L = \small {$1$}]{v3}
\Vertex[x = 0.45, y = 0.40, L = \small {$u_3$}]{v4}
\Vertex[x = 0.65, y = 0.40, L = \small {$1$}]{v5}
\Vertex[x = 0.75, y = 0.55, L = \small {$u_1$}]{v6}
\Vertex[x = 0.35, y = 0.25, L = \small {}]{v7}
\Vertex[x = 0.25, y = 0.40, L = \small {$3$}]{v8}
\Vertex[x = 0.55, y = 0.25, L = \small {$3$}]{v9}
\Vertex[x = 0.25, y = 0.70, L = \small {$2$}]{v10}
\Vertex[x = 0.35, y = 0.85, L = \small {$1$}]{v11}
\Vertex[x = 0.55, y = 0.85, L = \small {$2$}]{v12}
\Vertex[x = 0.85, y = 0.65, L = \small {}]{v13}
\Vertex[x = 0.85, y = 0.45, L = \small {}]{v14}
\Edge[](v0)(v1)
\Edge[](v0)(v2)
\Edge[](v0)(v4)
\Edge[](v0)(v5)
\Edge[](v0)(v6)
\Edge[](v0)(v3)
\Edge[](v3)(v1)
\Edge[](v3)(v6)
\Edge[](v6)(v5)
\Edge[](v2)(v1)
\Edge[](v2)(v4)
\Edge[](v4)(v5)
\Edge[](v7)(v9)
\Edge[](v7)(v8)
\Edge[](v7)(v4)
\Edge[](v8)(v4)
\Edge[](v9)(v4)
\Edge[](v5)(v9)
\Edge[](v2)(v8)
\Edge[](v1)(v12)
\Edge[](v1)(v11)
\Edge[](v1)(v10)
\Edge[](v10)(v2)
\Edge[](v10)(v11)
\Edge[](v12)(v11)
\Edge[](v12)(v3)
\Edge[](v14)(v6)
\Edge[](v13)(v6)
\Edge[](v14)(v5)
\Edge[](v14)(v13)
\Edge[](v13)(v3)
\end{tikzpicture}}}
\subfloat[A 6-vertex, $v$, with non-adjacent 6-neighbors $u_1$, $u_2$, and
$u_3$.]{\makebox[.45\textwidth]{
\begin{tikzpicture}[xscale = -9, yscale = 9]
\tikzstyle{VertexStyle}=[shape = circle, minimum size = 6pt, inner sep = 1.2pt, draw]
\Vertex[x = 0.550000011920929, y = 0.550000011920929, L = \small {$v$}]{v0}
\Vertex[x = 0.449999988079071, y = 0.699999988079071, L = \small {$u_2$}]{v1}
\Vertex[x = 0.349999994039536, y = 0.550000011920929, L = \small {$1$}]{v2}
\Vertex[x = 0.649999976158142, y = 0.699999988079071, L = \small {$1$}]{v3}
\Vertex[x = 0.449999988079071, y = 0.399999976158142, L = \small {$u_3$}]{v4}
\Vertex[x = 0.649999976158142, y = 0.399999976158142, L = \small {$1$}]{v5}
\Vertex[x = 0.75, y = 0.550000011920929, L = \small {$u_1$}]{v6}
\Vertex[x = 0.349999994039536, y = 0.25, L = \small {}]{v7}
\Vertex[x = 0.25, y = 0.399999976158142, L = \small {$3$}]{v8}
\Vertex[x = 0.550000011920929, y = 0.25, L = \small {$3$}]{v9}
\Vertex[x = 0.25, y = 0.699999988079071, L = \small {$2$}]{v10}
\Vertex[x = 0.349999994039536, y = 0.849999994039536, L = \small {$1$}]{v11}
\Vertex[x = 0.550000011920929, y = 0.849999994039536, L = \small {$2$}]{v12}
\Vertex[x = 0.850000023841858, y = 0.699999988079071, L = \small {$4$}]{v13}
\Vertex[x = 0.850000023841858, y = 0.399999976158142, L = \small {$4$}]{v14}
\Vertex[x = 0.949999988079071, y = 0.550000011920929, L = \small {}]{v15}
\Edge[](v0)(v1)
\Edge[](v0)(v2)
\Edge[](v0)(v4)
\Edge[](v0)(v5)
\Edge[](v0)(v6)
\Edge[](v0)(v3)
\Edge[](v3)(v1)
\Edge[](v3)(v6)
\Edge[](v6)(v5)
\Edge[](v2)(v1)
\Edge[](v2)(v4)
\Edge[](v4)(v5)
\Edge[](v7)(v9)
\Edge[](v7)(v8)
\Edge[](v7)(v4)
\Edge[](v8)(v4)
\Edge[](v9)(v4)
\Edge[](v5)(v9)
\Edge[](v2)(v8)
\Edge[](v1)(v12)
\Edge[](v1)(v11)
\Edge[](v1)(v10)
\Edge[](v10)(v2)
\Edge[](v10)(v11)
\Edge[](v12)(v11)
\Edge[](v12)(v3)
\Edge[](v14)(v6)
\Edge[](v14)(v5)
\Edge[](v13)(v3)
\Edge[](v15)(v13)
\Edge[](v15)(v6)
\Edge[](v15)(v14)
\Edge[](v13)(v6)
\end{tikzpicture}}}
\caption{The two cases of Lemma~\ref{alpha666}.}
\label{figalpha666}
\end{center}
\end{figure}

\begin{lem}
Every minimal $G$ has no 6-vertex $v$ with $6^-$-neighbors $u_1$, $u_2$, and
$u_3$ that are pairwise nonadjacent.
\label{alpha666}
\end{lem}
\begin{proof}
Lemma~\ref{CrunchToOneVertexNeighborhoodSize}, applied with
$J=\{u_1,u_2,u_3\}$, yields $12 \le \card{N(\set{u_1, u_2, u_3})} \le d(u_1) +
d(u_2) + d(u_3) - 5$.  Hence, by symmetry, assume that the vertices are
arranged as in Figure~\ref{figalpha666}(a) with all vertices distinct as drawn
or as in Figure~\ref{figalpha666}(b) with at most one pair of vertices
identified.

The first case is impossible by Lemma~\ref{ind-red} with $k=1$, using the
vertices labeled $2$ for $u_2$ and those labeled $3$ for $u_3$.  When the
vertices in Figure~\ref{figalpha666}(b) are distinct as drawn, we apply
Lemma~\ref{ind-red} with $k=0$, using the vertices labeled $2$ for $u_2$, the
vertices labeled $3$ for $u_3$, and those labeled $4$ for $u_1$.  Instead, by
symmetry and the fact that $G$ contains no separating 3-cycle, assume
that the vertices labeled $2$ and $3$ that are drawn at distance four are
identified; so $\card{N(\set{u_1, u_2, u_3})}=12$.
  Now the pairs of vertices labeled $1$ each have a common neighbor,
so the vertices labeled 1 must be an independent set, to avoid a separating
3-cycle.  Now, we apply Lemma~\ref{JIsThree}, using the
vertices labeled $4$ for the independent $2$-set.  This implies that 
$\card{N(\set{u_1, u_2, u_3})}\ge 13$, which contradicts our conclusion above
that $\card{N(\set{u_1, u_2, u_3})}=12$.
\end{proof}

\begin{figure}
\begin{center}
\subfloat[The forbidden configuration.]{\makebox[.45\textwidth]{
\begin{tikzpicture}[scale = 10]
\tikzstyle{VertexStyle}=[shape = circle, minimum size = 6pt, inner sep = 1.2pt, draw]
\Vertex[x = 0.45, y = 0.55, L = \small {}]{v0}
\Vertex[x = 0.45, y = 0.70, L = \small {$u_1$}]{v1}
\Vertex[x = 0.60, y = 0.65, L = \small {}]{v2}
\Vertex[x = 0.60, y = 0.50, L = \small {$u_2$}]{v3}
\Vertex[x = 0.60, y = 0.80, L = \small {}]{v4}
\Vertex[x = 0.75, y = 0.55, L = \small {}]{v5}
\Vertex[x = 0.75, y = 0.70, L = \small {$u_3$}]{v6}
\Vertex[x = 0.30, y = 0.65, L = \small {$z$}]{v7}
\Vertex[x = 0.35, y = 0.85, L = \small {$x$}]{v8}
\Vertex[x = 0.75, y = 0.85, L = \small {}]{v9}
\Vertex[x = 0.90, y = 0.80, L = \small {}]{v10}
\Vertex[x = 0.90, y = 0.65, L = \small {}]{v11}
\Vertex[x = 0.45, y = 0.40, L = \small {$w$}]{v12}
\Vertex[x = 0.55, y = 0.35, L = \small {}]{v13}
\Vertex[x = 0.65, y = 0.35, L = \small {}]{v14}
\Vertex[x = 0.75, y = 0.40, L = \small {$y$}]{v15}
\Edge[](v1)(v0)
\Edge[](v1)(v7)
\Edge[](v2)(v0)
\Edge[](v2)(v1)
\Edge[](v2)(v3)
\Edge[](v3)(v0)
\Edge[](v4)(v1)
\Edge[](v4)(v2)
\Edge[](v5)(v2)
\Edge[](v5)(v3)
\Edge[](v6)(v2)
\Edge[](v6)(v4)
\Edge[](v6)(v5)
\Edge[](v6)(v9)
\Edge[](v6)(v10)
\Edge[](v6)(v11)
\Edge[](v7)(v0)
\Edge[](v8)(v1)
\Edge[](v8)(v4)
\Edge[](v8)(v7)
\Edge[](v9)(v4)
\Edge[](v9)(v10)
\Edge[](v11)(v5)
\Edge[](v11)(v10)
\Edge[](v12)(v0)
\Edge[](v12)(v3)
\Edge[](v13)(v3)
\Edge[](v13)(v12)
\Edge[](v14)(v3)
\Edge[](v14)(v13)
\Edge[](v14)(v15)
\Edge[](v15)(v3)
\Edge[](v15)(v5)
\end{tikzpicture}}}
\subfloat[The $x \nonadj y$ case.]{\makebox[.45\textwidth]{
\begin{tikzpicture}[scale = 10]
\tikzstyle{VertexStyle}=[shape = circle, minimum size = 6pt, inner sep = 1.2pt, draw]
\Vertex[x = 0.45, y = 0.55, L = \small {$3,I_3$}]{v0}
\Vertex[x = 0.45, y = 0.70, L = \small {$1,I_1,I_2$}]{v1}
\Vertex[x = 0.60, y = 0.65, L = \small {$3$}]{v2}
\Vertex[x = 0.60, y = 0.50, L = \small {$2,I_1$}]{v3}
\Vertex[x = 0.60, y = 0.80, L = \small {$1$}]{v4}
\Vertex[x = 0.75, y = 0.55, L = \small {$3$}]{v5}
\Vertex[x = 0.75, y = 0.70, L = \small {$3,I_2,I_3$}]{v6}
\Vertex[x = 0.30, y = 0.65, L = \small {$3$}]{v7}
\Vertex[x = 0.35, y = 0.85, L = \small {$3,I_3$}]{v8}
\Vertex[x = 0.75, y = 0.85, L = \small {$1,I_1$}]{v9}
\Vertex[x = 0.90, y = 0.80, L = \small {$1$}]{v10}
\Vertex[x = 0.90, y = 0.65, L = \small {$1,I_1$}]{v11}
\Vertex[x = 0.45, y = 0.40, L = \small {$2,I_2$}]{v12}
\Vertex[x = 0.55, y = 0.35, L = \small {$2$}]{v13}
\Vertex[x = 0.65, y = 0.35, L = \small {$2,I_2$}]{v14}
\Vertex[x = 0.75, y = 0.40, L = \small {$3,I_3$}]{v15}
\Edge[](v1)(v0)
\Edge[](v1)(v7)
\Edge[](v2)(v0)
\Edge[](v2)(v1)
\Edge[](v2)(v3)
\Edge[](v3)(v0)
\Edge[](v4)(v1)
\Edge[](v4)(v2)
\Edge[](v5)(v2)
\Edge[](v5)(v3)
\Edge[](v6)(v2)
\Edge[](v6)(v4)
\Edge[](v6)(v5)
\Edge[](v6)(v9)
\Edge[](v6)(v10)
\Edge[](v6)(v11)
\Edge[](v7)(v0)
\Edge[](v8)(v1)
\Edge[](v8)(v4)
\Edge[](v8)(v7)
\Edge[](v9)(v4)
\Edge[](v9)(v10)
\Edge[](v11)(v5)
\Edge[](v11)(v10)
\Edge[](v12)(v0)
\Edge[](v12)(v3)
\Edge[](v13)(v3)
\Edge[](v13)(v12)
\Edge[](v14)(v3)
\Edge[](v14)(v13)
\Edge[](v14)(v15)
\Edge[](v15)(v3)
\Edge[](v15)(v5)
\end{tikzpicture}}}
~~~~~~~

\subfloat[The $x \adj y$ and $w \nonadj z$ case.]{\makebox[.45\textwidth]{
\begin{tikzpicture}[scale = 10]
\tikzstyle{VertexStyle}=[shape = circle, minimum size = 6pt, inner sep = 1.2pt, draw]
\Vertex[x = 0.35, y = 0.55, L = \small {$2$}]{v0}
\Vertex[x = 0.35, y = 0.70, L = \small {$1,I_1,I_3$}]{v1}
\Vertex[x = 0.50, y = 0.65, L = \small {$2$}]{v2}
\Vertex[x = 0.50, y = 0.50, L = \small {$2,I_3$}]{v3}
\Vertex[x = 0.50, y = 0.80, L = \small {$2,I_2$}]{v4}
\Vertex[x = 0.65, y = 0.55, L = \small {$2,I_2$}]{v5}
\Vertex[x = 0.65, y = 0.70, L = \small {$3,I_1$}]{v6}
\Vertex[x = 0.20, y = 0.65, L = \small {$2,I_2$}]{v7}
\Vertex[x = 0.25, y = 0.85, L = \small {$1$}]{v8}
\Vertex[x = 0.65, y = 0.85, L = \small {$3,I_3$}]{v9}
\Vertex[x = 0.80, y = 0.80, L = \small {$3$}]{v10}
\Vertex[x = 0.80, y = 0.65, L = \small {$3,I_3$}]{v11}
\Vertex[x = 0.35, y = 0.40, L = \small {$2,I_2$}]{v12}
\Vertex[x = 0.45, y = 0.35, L = \small {$1,I_1$}]{v13}
\Vertex[x = 0.55, y = 0.35, L = \small {$1$}]{v14}
\Vertex[x = 0.65, y = 0.40, L = \small {$1,I_1$}]{v15}
\Edge[](v1)(v0)
\Edge[](v1)(v7)
\Edge[](v2)(v0)
\Edge[](v2)(v1)
\Edge[](v2)(v3)
\Edge[](v3)(v0)
\Edge[](v4)(v1)
\Edge[](v4)(v2)
\Edge[](v5)(v2)
\Edge[](v5)(v3)
\Edge[](v6)(v2)
\Edge[](v6)(v4)
\Edge[](v6)(v5)
\Edge[](v6)(v9)
\Edge[](v6)(v10)
\Edge[](v6)(v11)
\Edge[](v7)(v0)
\Edge[](v8)(v1)
\Edge[](v8)(v4)
\Edge[](v8)(v7)
\Edge[](v9)(v4)
\Edge[](v9)(v10)
\Edge[](v11)(v5)
\Edge[](v11)(v10)
\Edge[](v12)(v0)
\Edge[](v12)(v3)
\Edge[](v13)(v3)
\Edge[](v13)(v12)
\Edge[](v14)(v3)
\Edge[](v14)(v13)
\Edge[](v14)(v15)
\Edge[](v15)(v3)
\Edge[](v15)(v5)
\end{tikzpicture}}}
\subfloat[The $x \adj y$ and $w \adj z$ case.]{\makebox[.45\textwidth]{
\begin{tikzpicture}[scale = 10]
\tikzstyle{VertexStyle}=[shape = circle, minimum size = 6pt, inner sep = 1.2pt, draw]
\Vertex[x = 0.30, y = 0.60, L = \small {$2,I_2$}]{v0}
\Vertex[x = 0.30, y = 0.75, L = \small {$1,I_1,I_3$}]{v1}
\Vertex[x = 0.45, y = 0.70, L = \small {$2$}]{v2}
\Vertex[x = 0.45, y = 0.55, L = \small {$2,I_3$}]{v3}
\Vertex[x = 0.45, y = 0.85, L = \small {$2,I_2$}]{v4}
\Vertex[x = 0.60, y = 0.60, L = \small {$2,I_2$}]{v5}
\Vertex[x = 0.60, y = 0.75, L = \small {$3,I_1$}]{v6}
\Vertex[x = 0.15, y = 0.70, L = \small {$1$}]{v7}
\Vertex[x = 0.20, y = 0.90, L = \small {$1$}]{v8}
\Vertex[x = 0.60, y = 0.90, L = \small {$3,I_3$}]{v9}
\Vertex[x = 0.75, y = 0.85, L = \small {$3$}]{v10}
\Vertex[x = 0.75, y = 0.70, L = \small {$3,I_3$}]{v11}
\Vertex[x = 0.30, y = 0.45, L = \small {$1,I_1$}]{v12}
\Vertex[x = 0.40, y = 0.40, L = \small {$2$}]{v13}
\Vertex[x = 0.50, y = 0.40, L = \small {$2,I_2$}]{v14}
\Vertex[x = 0.60, y = 0.45, L = \small {$1,I_1$}]{v15}
\Edge[](v1)(v0)
\Edge[](v1)(v7)
\Edge[](v2)(v0)
\Edge[](v2)(v1)
\Edge[](v2)(v3)
\Edge[](v3)(v0)
\Edge[](v4)(v1)
\Edge[](v4)(v2)
\Edge[](v5)(v2)
\Edge[](v5)(v3)
\Edge[](v6)(v2)
\Edge[](v6)(v4)
\Edge[](v6)(v5)
\Edge[](v6)(v9)
\Edge[](v6)(v10)
\Edge[](v6)(v11)
\Edge[](v7)(v0)
\Edge[](v8)(v1)
\Edge[](v8)(v4)
\Edge[](v8)(v7)
\Edge[](v9)(v4)
\Edge[](v9)(v10)
\Edge[](v11)(v5)
\Edge[](v11)(v10)
\Edge[](v12)(v0)
\Edge[](v12)(v3)
\Edge[](v13)(v3)
\Edge[](v13)(v12)
\Edge[](v14)(v3)
\Edge[](v14)(v13)
\Edge[](v14)(v15)
\Edge[](v15)(v3)
\Edge[](v15)(v5)
\end{tikzpicture}}}
~~~~~~~

\subfloat[The case of one identified pair.]{\makebox[.45\textwidth]{
\begin{tikzpicture}[scale = 10]
\tikzstyle{VertexStyle}=[shape = circle, minimum size = 6pt, inner sep = 1.2pt, draw]
\Vertex[x = 0.45, y = 0.55, L = \small {}]{v0}
\Vertex[x = 0.45, y = 0.70, L = \small {$u_1$}]{v1}
\Vertex[x = 0.60, y = 0.65, L = \small {}]{v2}
\Vertex[x = 0.60, y = 0.50, L = \small {$u_2$}]{v3}
\Vertex[x = 0.60, y = 0.80, L = \small {$1$}]{v4}
\Vertex[x = 0.75, y = 0.55, L = \small {}]{v5}
\Vertex[x = 0.75, y = 0.70, L = \small {$u_3$}]{v6}
\Vertex[x = 0.30, y = 0.65, L = \small {$1$}]{v7}
\Vertex[x = 0.35, y = 0.85, L = \small {$Q$}]{v8}
\Vertex[x = 0.75, y = 0.85, L = \small {$Q$}]{v9}
\Vertex[x = 0.90, y = 0.80, L = \small {$Q$}]{v10}
\Vertex[x = 0.90, y = 0.65, L = \small {$1$}]{v11}
\Vertex[x = 0.45, y = 0.40, L = \small {$2a$}]{v12}
\Vertex[x = 0.55, y = 0.35, L = \small {$2b$}]{v13}
\Vertex[x = 0.65, y = 0.35, L = \small {$2a$}]{v14}
\Vertex[x = 0.75, y = 0.40, L = \small {$2b$}]{v15}
\Edge[](v1)(v0)
\Edge[](v1)(v7)
\Edge[](v2)(v0)
\Edge[](v2)(v1)
\Edge[](v2)(v3)
\Edge[](v3)(v0)
\Edge[](v4)(v1)
\Edge[](v4)(v2)
\Edge[](v5)(v2)
\Edge[](v5)(v3)
\Edge[](v6)(v2)
\Edge[](v6)(v4)
\Edge[](v6)(v5)
\Edge[](v6)(v9)
\Edge[](v6)(v10)
\Edge[](v6)(v11)
\Edge[](v7)(v0)
\Edge[](v8)(v1)
\Edge[](v8)(v4)
\Edge[](v8)(v7)
\Edge[](v9)(v4)
\Edge[](v9)(v10)
\Edge[](v11)(v5)
\Edge[](v11)(v10)
\Edge[](v12)(v0)
\Edge[](v12)(v3)
\Edge[](v13)(v3)
\Edge[](v13)(v12)
\Edge[](v14)(v3)
\Edge[](v14)(v13)
\Edge[](v14)(v15)
\Edge[](v15)(v3)
\Edge[](v15)(v5)
\end{tikzpicture}}}
\caption{The case of Lemma \ref{alpha6567}.}
\label{fig6567}
\end{center}
\end{figure}
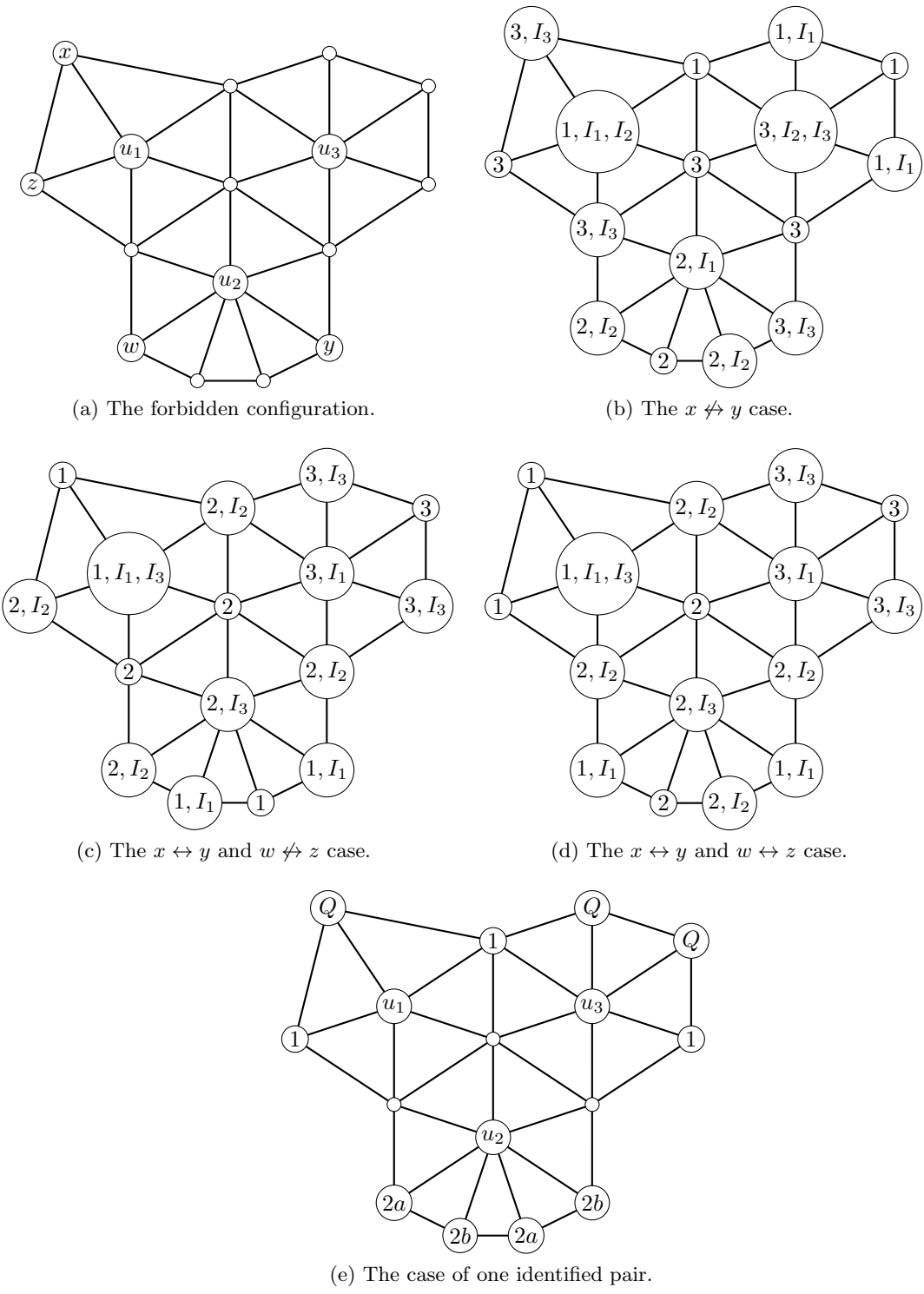
\begin{lem}\label{alpha6567}
Every minimal
$G$ has no 6-vertex $v$ with pairwise nonadjacent neighbors $u_1$, $u_2$, and
$u_3$, where $d(u_1)=5$, $d(u_2)\le 6$, and $d(u_3)=7$.
\end{lem}
\begin{proof}
Let $J = \set{u_1, u_2, u_3}$. By Lemma
\ref{CrunchToOneVertexNeighborhoodSize}, $12 \le |N(J)|\le 5+6+7-5=13$, so at most one pair
of vertices in Figure~\ref{fig6567}(a) are identified.

First, suppose the vertices in the figure are distinct as drawn.  Suppose $x
\nonadj y$, as in Figure~\ref{fig6567}(b).  
For each $i \in \set{1,2,3}$, let $S_i$ consist of the vertices labeled $i$. 
Now for each $i \in \set{1,2,3}$, $G[S_i]$ is connected.  Clearly, for each $i
\in \set{1,2}$ the vertices labeled $I_i$ form an independent $4$-set.  Since
$x \nonadj y$, the vertices labeled $I_3$ also form an independent $4$-set.  Note
that $S_1 \adj S_3$ and $S_2 \adj S_3$; however, possibly $S_1 \nonadj S_2$. 
If $S_1 \adj S_2$, then we can apply Lemma~\ref{ThreeHelperFourteen} to get a
contradiction.  So, we assume that $S_1\nonadj S_2$.  But now we have an
independent $5$-set consisting of $u_1$, the two vertices labeled $\set{1,
I_1}$ and the two vertices labeled $\set{2,I_2}$; hence $\alpha(G[S_1 \cup S_2
\cup J]) \ge 5$.  So, we can apply Lemma~\ref{uber-red} to get a contradiction.
So, instead we assume $x \adj y$.  

Suppose $w \nonadj z$, as in Figure~\ref{fig6567}(c).  
For each $i \in \set{1,2,3}$, let $S_i$ consist of the vertices labeled $i$ .
Clearly $G[S_i]$ is connected for each $i \in \set{2,3}$.  Also, $G[S_1]$ is
connected because $x \adj y$.  Note that for each $i \in \set{1,3}$, the 
vertices labeled $I_i$ form an independent $4$-set.
Since $x \adj y$ and $w \nonadj z$, the vertices labeled $I_2$ also form an
independent $4$-set.  Note that $S_1 \adj S_2$ and $S_2
\adj S_3$; however, possibly $S_1 \nonadj S_3$.  If $S_1 \adj S_3$, then we
apply Lemma~\ref{ThreeHelperFourteen} to get a contradiction.  So instead we 
assume that $S_1\nonadj S_3$.  But now we again have an independent 5-set,
consisting of $u_1$, the two vertices labeled $\set{1, I_1}$, and the two vertices
labeled $\set{3,I_3}$; hence $\alpha(G[S_1 \cup S_3 \cup J]) \ge 5$.  So, again we
apply Lemma~\ref{uber-red} to get a contradiction.
Thus, we instead assume $w \adj z$.

Now consider Figure~\ref{fig6567}(d).  
For each $i \in \set{1,2,3}$, let $S_i$ consist of the vertices labeled $i$.
Note that $G[S_i]$ is connected for each $i \in \set{2,3}$. 
Also, $G[S_1]$ is connected because $x \adj y$ and $w \adj z$. Clearly, the
vertices labeled $I_i$ form an independent $4$-set for each $i \in \set{1,3}$. 
Since $x \adj y$, the vertices labeled $I_2$ also form an independent $4$-set. 
Note that $S_1 \adj S_2$ and $S_2 \adj S_3$; however, possibly
$S_1 \nonadj S_3$.  If $S_1 \adj S_3$, then we apply Lemma
\ref{ThreeHelperFourteen} to get a contradiction.  So, instead we assume that
$S_1\nonadj S_3$.  But now we have an independent $5$-set, consisting of $u_1$,
the two vertices labeled $\set{1, I_1}$, and the two vertices labeled
$\set{3,I_3}$; hence $\alpha(G[S_1 \cup S_3 \cup J]) \ge 5$.  So, we
apply Lemma~\ref{uber-red} to get a contradiction.

Hence, we may assume that exactly one pair of vertices in
Figure~\ref{fig6567}(a) is identified. No neighbor of $u_1$ can be identified
with a neighbor of $u_3$, since then $u_1$ and $u_3$ would have three common
neighbors, violating Lemma~\ref{JIsTwo}.  Hence, to avoid separating 3-cycles,
we assume that a vertex labeled $2a$ is identified with a vertex labeled $Q$
(the case where a vertex labeled $2b$ is identified with a vertex lableled $Q$
is nearly identical, so we omit the details).  But now the rightmost vertex
labeled $1$ and the leftmost vertex labeled $1$ are on opposite sides of a
separating cycle and hence nonadjacent.  Therefore, $u_2$ together with the
vertices labeled $1$ is an independent $4$-set.  So, now we apply
Lemma~\ref{JIsThree} to get a contradiction, using the vertices labeled $2b$
for the independent $2$-set.
\end{proof}

\begin{figure}
\begin{center}
\subfloat[Here $u_2$ and $u_3$ have a common neighbor in
$N(u_1)$.]{\makebox[.42\textwidth]{
\begin{tikzpicture}[scale = 9]
\tikzstyle{VertexStyle}=[shape = circle, minimum size = 6pt, inner sep = 1.2pt, draw]
\Vertex[x = 1.05, y = 0.65, L = \small {$u_3$}]{v0}
\Vertex[x = 0.90, y = 0.75, L = \small {}]{v1}
\Vertex[x = 0.90, y = 0.55, L = \small {}]{v2}
\Vertex[x = 0.85, y = 0.40, L = \small {$u_2$}]{v3}
\Vertex[x = 0.60, y = 0.75, L = \small {}]{v4}
\Vertex[x = 0.60, y = 0.55, L = \small {$1$}]{v5}
\Vertex[x = 0.75, y = 0.65, L = \small {$u_1$}]{v6}
\Vertex[x = 0.70, y = 0.35, L = \small {$2$}]{v7}
\Vertex[x = 1.00, y = 0.40, L = \small {$2$}]{v8}
\Vertex[x = 0.90, y = 0.25, L = \small {}]{v9}
\Vertex[x = 1.05, y = 0.80, L = \small {$3$}]{v10}
\Vertex[x = 1.20, y = 0.75, L = \small {$3b$}]{v11}
\Vertex[x = 1.20, y = 0.55, L = \small {$3$}]{v12}
\Vertex[x = 1.05, y = 0.5, L = \small {$3b$}]{v13}
\Vertex[x = 0.75, y = 0.5, L = \small {}]{v14}
\Vertex[x = 0.75, y = 0.80, L = \small {$1$}]{v15}
\Edge[](v1)(v0)
\Edge[](v2)(v0)
\Edge[](v2)(v1)
\Edge[](v3)(v14)
\Edge[](v5)(v4)
\Edge[](v6)(v1)
\Edge[](v6)(v2)
\Edge[](v6)(v4)
\Edge[](v6)(v5)
\Edge[](v7)(v3)
\Edge[](v7)(v14)
\Edge[](v8)(v3)
\Edge[](v9)(v3)
\Edge[](v9)(v7)
\Edge[](v9)(v8)
\Edge[](v10)(v0)
\Edge[](v10)(v1)
\Edge[](v11)(v0)
\Edge[](v11)(v10)
\Edge[](v12)(v0)
\Edge[](v12)(v11)
\Edge[](v13)(v0)
\Edge[](v13)(v2)
\Edge[](v13)(v12)
\Edge[](v14)(v2)
\Edge[](v14)(v5)
\Edge[](v14)(v6)
\Edge[](v15)(v1)
\Edge[](v15)(v4)
\Edge[](v15)(v6)
\Edge[](v8)(v2)
\Edge[](v3)(v2)
\end{tikzpicture}}}

\subfloat[Here $u_2$ and $u_3$ have adjacent neighbors in
$N(u_1)$.]{\makebox[.42\textwidth]{
\begin{tikzpicture}[scale = 9]
\tikzstyle{VertexStyle}=[shape = circle, minimum size = 6pt, inner sep = 1.2pt, draw]
\Vertex[x = 1.05, y = 0.65, L = \small {$u_3$}]{v0}
\Vertex[x = 0.90, y = 0.75, L = \small {}]{v1}
\Vertex[x = 0.90, y = 0.55, L = \small {$1a, 1b$}]{v2}
\Vertex[x = 0.60, y = 0.40, L = \small {$u_2$}]{v3}
\Vertex[x = 0.60, y = 0.75, L = \small {$1b, 5a$}]{v4}
\Vertex[x = 0.60, y = 0.55, L = \small {}]{v5}
\Vertex[x = 0.75, y = 0.65, L = \small {$u_1$}]{v6}
\Vertex[x = 0.70, y = 0.30, L = \small {$2$}]{v7}
\Vertex[x = 0.45, y = 0.50, L = \small {$2, 5a, 5b$}]{v8}
\Vertex[x = 0.50, y = 0.30, L = \small {}]{v9}
\Vertex[x = 1.05, y = 0.80, L = \small {$1a, 1b, 4$}]{v10}
\Vertex[x = 1.20, y = 0.75, L = \small {$3$}]{v11}
\Vertex[x = 1.20, y = 0.55, L = \small {$4$}]{v12}
\Vertex[x = 1.05, y = 0.50, L = \small {$3$}]{v13}
\Vertex[x = 0.75, y = 0.50, L = \small {$5a, 5b$}]{v14}
\Vertex[x = 0.75, y = 0.80, L = \small {$1a, 5b$}]{v15}
\Edge[](v1)(v0)
\Edge[](v2)(v0)
\Edge[](v2)(v1)
\Edge[](v3)(v14)
\Edge[](v5)(v3)
\Edge[](v5)(v4)
\Edge[](v6)(v1)
\Edge[](v6)(v2)
\Edge[](v6)(v4)
\Edge[](v6)(v5)
\Edge[](v7)(v3)
\Edge[](v7)(v14)
\Edge[](v8)(v3)
\Edge[](v8)(v5)
\Edge[](v9)(v3)
\Edge[](v9)(v7)
\Edge[](v9)(v8)
\Edge[](v10)(v0)
\Edge[](v10)(v1)
\Edge[](v11)(v0)
\Edge[](v11)(v10)
\Edge[](v12)(v0)
\Edge[](v12)(v11)
\Edge[](v13)(v0)
\Edge[](v13)(v2)
\Edge[](v13)(v12)
\Edge[](v14)(v2)
\Edge[](v14)(v5)
\Edge[](v14)(v6)
\Edge[](v15)(v1)
\Edge[](v15)(v4)
\Edge[](v15)(v6)
\end{tikzpicture}}}

\subfloat[Here $u_2$ and $u_3$ have neighbors at
distance 2\ in $N(u_1)$.]{\makebox[.5\textwidth]{
\begin{tikzpicture}[scale = 9]
\tikzstyle{VertexStyle}=[shape = circle, minimum size = 6pt, inner sep = 1.2pt, draw]
\Vertex[x = 0.90, y = 0.65, L = \small {$u_3$}]{v0}
\Vertex[x = 0.75, y = 0.75, L = \small {}]{v1}
\Vertex[x = 0.75, y = 0.55, L = \small {}]{v2}
\Vertex[x = 0.30, y = 0.65, L = \small {$u_2$}]{v3}
\Vertex[x = 0.45, y = 0.75, L = \small {}]{v4}
\Vertex[x = 0.45, y = 0.55, L = \small {}]{v5}
\Vertex[x = 0.60, y = 0.65, L = \small {$u_1$}]{v6}
\Vertex[x = 0.30, y = 0.80, L = \small {$2$}]{v7}
\Vertex[x = 0.30, y = 0.50, L = \small {$2$}]{v8}
\Vertex[x = 0.15, y = 0.65, L = \small {}]{v9}
\Vertex[x = 0.90, y = 0.80, L = \small {$3$}]{v10}
\Vertex[x = 1.05, y = 0.75, L = \small {$3b$}]{v11}
\Vertex[x = 1.05, y = 0.55, L = \small {$3$}]{v12}
\Vertex[x = 0.90, y = 0.50, L = \small {$3b$}]{v13}
\Vertex[x = 0.60, y = 0.50, L = \small {$1$}]{v14}
\Vertex[x = 0.60, y = 0.80, L = \small {$1$}]{v15}
\Edge[](v1)(v0)
\Edge[](v2)(v0)
\Edge[](v2)(v1)
\Edge[](v4)(v3)
\Edge[](v5)(v3)
\Edge[](v5)(v4)
\Edge[](v6)(v1)
\Edge[](v6)(v2)
\Edge[](v6)(v4)
\Edge[](v6)(v5)
\Edge[](v7)(v3)
\Edge[](v7)(v4)
\Edge[](v8)(v5)
\Edge[](v8)(v3)
\Edge[](v9)(v7)
\Edge[](v9)(v8)
\Edge[](v9)(v3)
\Edge[](v10)(v1)
\Edge[](v10)(v0)
\Edge[](v11)(v10)
\Edge[](v11)(v0)
\Edge[](v12)(v11)
\Edge[](v12)(v0)
\Edge[](v13)(v12)
\Edge[](v13)(v0)
\Edge[](v13)(v2)
\Edge[](v15)(v4)
\Edge[](v15)(v6)
\Edge[](v15)(v1)
\Edge[](v14)(v5)
\Edge[](v14)(v2)
\Edge[](v14)(v6)
\end{tikzpicture}}}

\caption{
The cases of Lemma~\ref{alpha665big}. The three possibilities for an
independent 3-set $\set{u_1,u_2,u_3}$ where $d(u_1)=6$, $d(u_2)\le 6$, $d(u_3)=5$, and
each of $u_2$ and $u_3$ has two neighbors in common with $u_1$.
\label{figalpha665big}
}
\end{center}
\end{figure}
\begin{lem}
Let $u_1$ be a $6$-vertex with nonadjacent vertices $u_2$ and $u_3$ each at
distance two from $u_1$, where $u_2$ is a $5$-vertex and $u_3$ is a
$6^-$-vertex. A minimal $G$ cannot have $u_1$ and $u_2$ with two common
neighbors, and also $u_1$ and $u_3$ with two common neighbors.
\label{alpha665big}
\end{lem}
\begin{proof}
Figure~\ref{figalpha665big} shows the possible arrangements when $u_3$ is a
6-vertex. The case when $u_3$ is a 5-vertex is similar, but
easier.  In particular, when $u_3$ is a 5-vertex, we already know
$|N(\set{u_1,u_2,u_3})|\le 12$, so all vertices in the corresponding figures
must be distinct as drawn.  Furthermore, it now suffices to apply
Lemma~\ref{ind-red} with $k=1$.  We omit further details.  So suppose instead
that $d(u_3)=6$.

First, suppose all vertices in the figures are distinct as drawn.  Now
Figures~\ref{figalpha665big}(a,c) are impossible by Lemma~\ref{ind-red} with
$k=0$; for each $i \in \set{1,2,3}$, we use the vertices labeled $i$ as the
independent $2$-set for $u_i$.  For Figure~\ref{figalpha665big}(b), let $I_1$
be the vertices labeled $u_2$ or $1a$ and let $I_2$ be the vertices labeled
$u_2$ or $1b$.  To avoid a separating 3-cycle, at least one of $I_1$ or $I_2$
is independent.  Hence Figure~\ref{figalpha665big}(b) is impossible by Lemma
\ref{JIsThreeFourteen}; for the independent 4-set, use $I_1$ or $I_2$ and
for each $i\in\set{2,3}$, use the vertices labeled
$i$ as the independent 2-set for $u_i$.

By Lemma~\ref{CrunchToOneVertexNeighborhoodSize}, $|N(J)| \ge 12$, so exactly
one pair of vertices is 
identified in one of Figures~\ref{figalpha665big}(a,b,c).
First, consider Figures~\ref{figalpha665big}(a,c) simultaneously. 
Since $G$ has no separating 3-cycle, the identified pair must contain a vertex
labeled 3.  Now we apply Lemma~\ref{ind-red} with $k=1$, using the vertices
labeled $3b$ in place of those labeled $3$.

Finally, for Figure~\ref{figalpha665big}(b),  we apply Lemma~\ref{JIsThree}. 
For the independent 2-set we use either the vertices labeled 3 or the
vertices labeled 4; at least one of these pairs contains no identified vertex. 
For the independent 4-set, we use either $u_3$ and the vertices labeled $5a$
or else $u_3$ and the vertices labeled $5b$.  Since $G$ has no separating
3-cycle, at least one of these 4-sets will be independent.
\end{proof}

\begin{lem}
Every minimal $G$ has no 7-vertex $v$ with a 5-neighbor and two other $6^-$-neighbors,
$u_1$, $u_2$, and $u_3$, that are pairwise nonadjacent.
In other words, Figures~\ref{fig7566}(a--e) are forbidden.
\label{alpha7566}
\end{lem}
\begin{figure}
\begin{center}
\subfloat[A 7-vertex, $v$, with non-adjacent 5-neighbors, $u_1$, $u_2$, and
$u_3$.]{\makebox[.45\textwidth]{
\begin{tikzpicture}[scale = 10]
\tikzstyle{VertexStyle}=[shape = circle, minimum size = 6pt, inner sep = 1.2pt, draw]
\Vertex[x = 0.55, y = 0.55, L = \small {$v$}]{v0}
\Vertex[x = 0.40, y = 0.65, L = \small {}]{v1}
\Vertex[x = 0.35, y = 0.50, L = \small {$u_1$}]{v2}
\Vertex[x = 0.70, y = 0.65, L = \small {}]{v3}
\Vertex[x = 0.45, y = 0.40, L = \small {}]{v4}
\Vertex[x = 0.65, y = 0.40, L = \small {}]{v5}
\Vertex[x = 0.75, y = 0.50, L = \small {$u_3$}]{v6}
\Vertex[x = 0.55, y = 0.75, L = \small {$u_2$}]{v7}
\Vertex[x = 0.85, y = 0.40, L = \small {}]{v8}
\Vertex[x = 0.85, y = 0.60, L = \small {}]{v9}
\Vertex[x = 0.25, y = 0.60, L = \small {}]{v10}
\Vertex[x = 0.25, y = 0.40, L = \small {}]{v11}
\Vertex[x = 0.40, y = 0.85, L = \small {}]{v13}
\Vertex[x = 0.70, y = 0.85, L = \small {}]{v14}
\Edge[](v0)(v1)
\Edge[](v0)(v2)
\Edge[](v0)(v4)
\Edge[](v0)(v5)
\Edge[](v0)(v6)
\Edge[](v0)(v3)
\Edge[](v3)(v6)
\Edge[](v6)(v5)
\Edge[](v2)(v1)
\Edge[](v2)(v4)
\Edge[](v4)(v5)
\Edge[](v7)(v0)
\Edge[](v7)(v3)
\Edge[](v7)(v1)
\Edge[](v9)(v6)
\Edge[](v9)(v8)
\Edge[](v8)(v5)
\Edge[](v8)(v6)
\Edge[](v9)(v3)
\Edge[](v11)(v10)
\Edge[](v11)(v2)
\Edge[](v11)(v4)
\Edge[](v10)(v2)
\Edge[](v10)(v1)
\Edge[](v13)(v14)
\Edge[](v13)(v7)
\Edge[](v14)(v7)
\Edge[](v1)(v13)
\Edge[](v3)(v14)
\end{tikzpicture}}}
~~~~~~
\subfloat[A 7-vertex, $v$, with a 6-neighbor, $u_3$, and
two 5-neighbors, $u_1$ and $u_2$. 
]{\makebox[.45\textwidth]{
\begin{tikzpicture}[xscale = -10, yscale = 10]
\tikzstyle{VertexStyle}=[shape = circle, minimum size = 6pt, inner sep = 1.2pt, draw]
\Vertex[x = 0.55, y = 0.55, L = \small {$v$}]{v0}
\Vertex[x = 0.40, y = 0.65, L = \small {}]{v1}
\Vertex[x = 0.35, y = 0.50, L = \small {$u_3$}]{v2}
\Vertex[x = 0.70, y = 0.65, L = \small {}]{v3}
\Vertex[x = 0.45, y = 0.40, L = \small {}]{v4}
\Vertex[x = 0.65, y = 0.40, L = \small {$1$}]{v5}
\Vertex[x = 0.75, y = 0.50, L = \small {$u_1$}]{v6}
\Vertex[x = 0.55, y = 0.75, L = \small {$u_2$}]{v7}
\Vertex[x = 0.85, y = 0.40, L = \small {}]{v8}
\Vertex[x = 0.85, y = 0.60, L = \small {$1$}]{v9}
\Vertex[x = 0.25, y = 0.60, L = \small {$2$}]{v10}
\Vertex[x = 0.25, y = 0.40, L = \small {$2$}]{v11}
\Vertex[x = 0.15, y = 0.50, L = \small {}]{v12}
\Vertex[x = 0.40, y = 0.85, L = \small {}]{v13}
\Vertex[x = 0.70, y = 0.85, L = \small {}]{v14}
\Edge[](v0)(v1)
\Edge[](v0)(v2)
\Edge[](v0)(v4)
\Edge[](v0)(v5)
\Edge[](v0)(v6)
\Edge[](v0)(v3)
\Edge[](v3)(v6)
\Edge[](v6)(v5)
\Edge[](v2)(v1)
\Edge[](v2)(v4)
\Edge[](v4)(v5)
\Edge[](v7)(v0)
\Edge[](v7)(v3)
\Edge[](v7)(v1)
\Edge[](v9)(v6)
\Edge[](v9)(v8)
\Edge[](v8)(v5)
\Edge[](v8)(v6)
\Edge[](v9)(v3)
\Edge[](v11)(v2)
\Edge[](v11)(v4)
\Edge[](v10)(v2)
\Edge[](v10)(v1)
\Edge[](v12)(v10)
\Edge[](v12)(v2)
\Edge[](v12)(v11)
\Edge[](v13)(v14)
\Edge[](v13)(v7)
\Edge[](v14)(v7)
\Edge[](v1)(v13)
\Edge[](v3)(v14)
\end{tikzpicture}}}

\subfloat[A 7-vertex, $v$, with a 6-neighbor, $u_2$, and
two 5-neighbors, $u_1$ and $u_3$. 
]{\makebox[.45\textwidth]{
\begin{tikzpicture}[scale = 10]
\tikzstyle{VertexStyle}=[shape = circle, minimum size = 6pt, inner sep = 1.2pt, draw]
\Vertex[x = 0.55, y = 0.55, L = \small {$v$}]{v0}
\Vertex[x = 0.40, y = 0.65, L = \small {}]{v1}
\Vertex[x = 0.35, y = 0.50, L = \small {$u_1$}]{v2}
\Vertex[x = 0.70, y = 0.65, L = \small {}]{v3}
\Vertex[x = 0.45, y = 0.40, L = \small {$1$}]{v4}
\Vertex[x = 0.65, y = 0.40, L = \small {}]{v5}
\Vertex[x = 0.75, y = 0.50, L = \small {$u_3$}]{v6}
\Vertex[x = 0.55, y = 0.75, L = \small {$u_2$}]{v7}
\Vertex[x = 0.25, y = 0.60, L = \small {$1$}]{v8}
\Vertex[x = 0.25, y = 0.40, L = \small {}]{v9}
\Vertex[x = 0.40, y = 0.85, L = \small {$2$}]{v11}
\Vertex[x = 0.70, y = 0.85, L = \small {$2$}]{v12}
\Vertex[x = 0.85, y = 0.40, L = \small {}]{v13}
\Vertex[x = 0.85, y = 0.60, L = \small {}]{v14}
\Vertex[x = 0.55, y = 0.95, L = \small {}]{v16}
\Edge[](v0)(v1)
\Edge[](v0)(v2)
\Edge[](v0)(v4)
\Edge[](v0)(v5)
\Edge[](v0)(v6)
\Edge[](v0)(v3)
\Edge[](v3)(v6)
\Edge[](v6)(v5)
\Edge[](v2)(v1)
\Edge[](v2)(v4)
\Edge[](v4)(v5)
\Edge[](v7)(v0)
\Edge[](v7)(v3)
\Edge[](v7)(v1)
\Edge[](v9)(v2)
\Edge[](v9)(v4)
\Edge[](v8)(v2)
\Edge[](v8)(v1)
\Edge[](v8)(v9)
\Edge[](v12)(v3)
\Edge[](v12)(v7)
\Edge[](v11)(v1)
\Edge[](v11)(v7)
\Edge[](v14)(v6)
\Edge[](v14)(v13)
\Edge[](v13)(v5)
\Edge[](v13)(v6)
\Edge[](v14)(v3)
\Edge[](v11)(v16)
\Edge[](v16)(v7)
\Edge[](v16)(v12)
\end{tikzpicture}}}
\subfloat[A 7-vertex, $v$, with a 5-neighbor, $u_1$, and
two 6-neighbors, $u_2$ and $u_3$. 
]{\makebox[.45\textwidth]{
\begin{tikzpicture}[xscale = 10,yscale=10]
\tikzstyle{VertexStyle}=[shape = circle, minimum size = 6pt, inner sep = 1.2pt, draw]
\Vertex[x = 0.55, y = 0.55, L = \small {$v$}]{v0}
\Vertex[x = 0.40, y = 0.65, L = \small {$4$}]{v1}
\Vertex[x = 0.35, y = 0.50, L = \small {$u_1$}]{v2}
\Vertex[x = 0.70, y = 0.65, L = \small {$4$}]{v3}
\Vertex[x = 0.45, y = 0.40, L = \small {$1$}]{v4}
\Vertex[x = 0.65, y = 0.40, L = \small {$4$}]{v5}
\Vertex[x = 0.75, y = 0.50, L = \small {$u_3$}]{v6}
\Vertex[x = 0.55, y = 0.75, L = \small {$u_2$}]{v7}
\Vertex[x = 0.25, y = 0.60, L = \small {$1$}]{v8}
\Vertex[x = 0.25, y = 0.40, L = \small {}]{v9}
\Vertex[x = 0.40, y = 0.85, L = \small {$2$}]{v11}
\Vertex[x = 0.70, y = 0.85, L = \small {$2$}]{v12}
\Vertex[x = 0.85, y = 0.40, L = \small {$3$}]{v13}
\Vertex[x = 0.85, y = 0.60, L = \small {$3$}]{v14}
\Vertex[x = 0.95, y = 0.50, L = \small {$4$}]{v15}
\Vertex[x = 0.55, y = 0.95, L = \small {}]{v16}
\Edge[](v0)(v1)
\Edge[](v0)(v2)
\Edge[](v0)(v4)
\Edge[](v0)(v5)
\Edge[](v0)(v6)
\Edge[](v0)(v3)
\Edge[](v3)(v6)
\Edge[](v6)(v5)
\Edge[](v2)(v1)
\Edge[](v2)(v4)
\Edge[](v4)(v5)
\Edge[](v7)(v0)
\Edge[](v7)(v3)
\Edge[](v7)(v1)
\Edge[](v9)(v2)
\Edge[](v9)(v4)
\Edge[](v8)(v2)
\Edge[](v8)(v1)
\Edge[](v8)(v9)
\Edge[](v12)(v3)
\Edge[](v12)(v7)
\Edge[](v11)(v1)
\Edge[](v11)(v7)
\Edge[](v14)(v6)
\Edge[](v14)(v15)
\Edge[](v13)(v5)
\Edge[](v13)(v6)
\Edge[](v13)(v15)
\Edge[](v15)(v6)
\Edge[](v14)(v3)
\Edge[](v11)(v16)
\Edge[](v16)(v7)
\Edge[](v16)(v12)
\end{tikzpicture}}}

\subfloat[A 7-vertex, $v$, with a 5-neighbor, $u_2$, and
two 6-neighbors, $u_1$ and $u_3$. 
]{\makebox[.45\textwidth]{
\begin{tikzpicture}[scale = 10]
\tikzstyle{VertexStyle}=[shape = circle, minimum size = 6pt, inner sep = 1.2pt, draw]
\Vertex[x = 0.55, y = 0.55, L = \small {$v$}]{v0}
\Vertex[x = 0.40, y = 0.65, L = \small {}]{v1}
\Vertex[x = 0.35, y = 0.50, L = \small {$u_1$}]{v2}
\Vertex[x = 0.70, y = 0.65, L = \small {}]{v3}
\Vertex[x = 0.45, y = 0.40, L = \small {}]{v4}
\Vertex[x = 0.65, y = 0.40, L = \small {}]{v5}
\Vertex[x = 0.75, y = 0.50, L = \small {$u_3$}]{v6}
\Vertex[x = 0.55, y = 0.75, L = \small {$u_2$}]{v7}
\Vertex[x = 0.25, y = 0.60, L = \small {$1$}]{v8}
\Vertex[x = 0.25, y = 0.40, L = \small {$1$}]{v9}
\Vertex[x = 0.15, y = 0.50, L = \small {}]{v10}
\Vertex[x = 0.40, y = 0.85, L = \small {}]{v11}
\Vertex[x = 0.70, y = 0.85, L = \small {}]{v12}
\Vertex[x = 0.85, y = 0.40, L = \small {$2$}]{v13}
\Vertex[x = 0.85, y = 0.60, L = \small {$2$}]{v14}
\Vertex[x = 0.95, y = 0.5, L = \small {}]{v15}
\Edge[](v0)(v1)
\Edge[](v0)(v2)
\Edge[](v0)(v4)
\Edge[](v0)(v5)
\Edge[](v0)(v6)
\Edge[](v0)(v3)
\Edge[](v3)(v6)
\Edge[](v6)(v5)
\Edge[](v2)(v1)
\Edge[](v2)(v4)
\Edge[](v4)(v5)
\Edge[](v7)(v0)
\Edge[](v7)(v3)
\Edge[](v7)(v1)
\Edge[](v9)(v2)
\Edge[](v9)(v4)
\Edge[](v8)(v2)
\Edge[](v8)(v1)
\Edge[](v10)(v8)
\Edge[](v10)(v2)
\Edge[](v10)(v9)
\Edge[](v12)(v3)
\Edge[](v12)(v7)
\Edge[](v11)(v1)
\Edge[](v11)(v7)
\Edge[](v12)(v11)
\Edge[](v14)(v6)
\Edge[](v14)(v15)
\Edge[](v13)(v5)
\Edge[](v13)(v6)
\Edge[](v13)(v15)
\Edge[](v15)(v6)
\Edge[](v14)(v3)
\end{tikzpicture}}}

\caption{The five cases of Lemma~\ref{alpha7566}.}
\label{fig7566}
\end{center}
\end{figure}
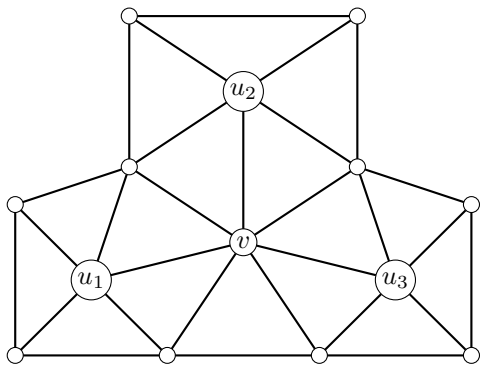
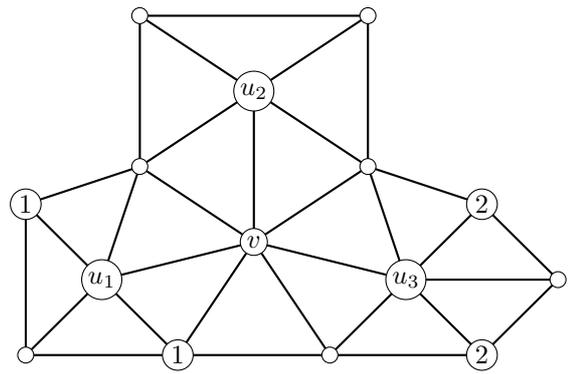
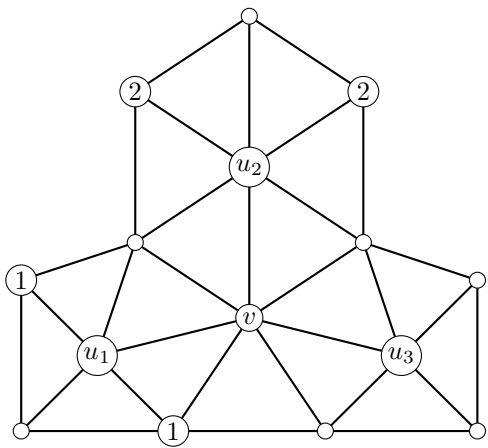
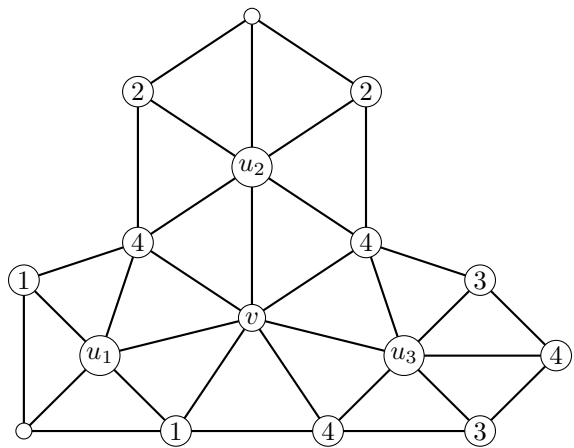
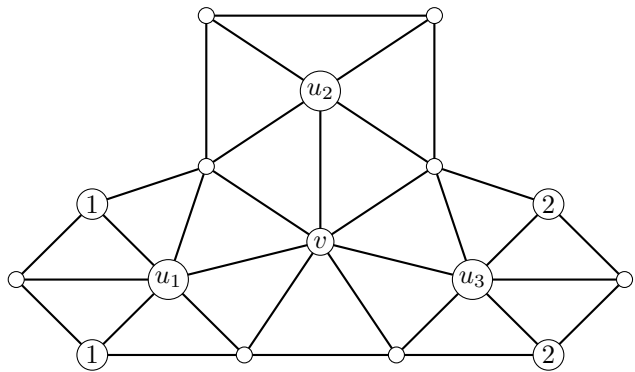
\begin{proof}
Lemma~\ref{CrunchToOneVertexNeighborhoodSize} yields $12 \le
\card{N(\set{u_1, u_2, u_3})} \le d(u_1) + d(u_2) + d(u_3) - 4\le 5+6+6-4\le
13$.  
In Figure~\ref{fig7566}(a), $\card{N(\set{u_1, u_2, u_3})} = 11$. 
So, by symmetry, we assume that the vertices are arranged as in
Figures~\ref{fig7566}(b,c) with all vertices distinct as drawn or as in
Figures~\ref{fig7566}(d,e) with at most one pair of vertices identified.

First suppose the vertices are disinct as drawn.  For
Figures~\ref{fig7566}(b,c,d), we apply Lemma~\ref{ind-red}; for (b) and (c) we
use $k=1$, and for (d) we use $k=0$.
For Figure~\ref{fig7566}(e), we apply Lemma~\ref{JIsThreeFourteenWithSeven},
using the vertices labeled $1$ for $M_1$ and the those labeled $2$ for $M_2$.  
Now $\card{N(\{u_1,u_2,u_3\})}\ge 14$ is a contradiction.

So, instead suppose that a single pair of vertices is identified in one of
Figures~\ref{fig7566}(d,e). 
First consider (d).
If a vertex labeled $1$ is identified with another vertex, then we apply 
Lemma~\ref{JIsThreeWithSeven} using the vertices labeled $2$ 
for the independent $2$-set (vertices labeled 1 and 2 cannot be identified,
since they are drawn at distance at most 3).  Otherwise, the identified
vertices must be those labeled $2$ and $3$ that are drawn at distance four.  
Now the vertices labeled $4$ are pairwise at distance two, so must be
an independent 4-set.  Now we get a contradiction, by applying Lemma
\ref{JIsThree} using the vertices labeled $1$ for the independent $2$-set.

Finally, consider Figure~\ref{fig7566}(e), with a single pair of vertices
identified.  Again we apply Lemma~\ref{ind-red}, with $k=1$. Since
$u_1$ has three possibilities for its pair of nonadjacent neighbors, and no
neighbor of $u_1$ appears in all three of these pairs,  $u_1$ satisfies
condition (2).  Similarly, $u_3$ also satisfies condition (2).
\end{proof}

\begin{lem}
Let $v_1$, $v_2$, $v_3$ be the corners of a 3-face, each a $6^+$-vertex.  Let
$u_1$, $u_2$, $u_3$ be the other pairwise common neighbors of $v_1$, $v_2$,
$v_3$, i.e., $u_1$ is adjacent to $v_1$ and $v_2$, $u_2$ is adjacent to $v_2$ 
and $v_3$, and $u_3$ is adjacent to $v_3$ and $v_1$.  
We cannot have $|N(\set{u_1,u_2,u_3})|\le 13$.  
In particular, we cannot have $d(u_1)=d(u_2)=5$ and $d(u_3)\le 6$.
\label{alpha556tri}
\end{lem}
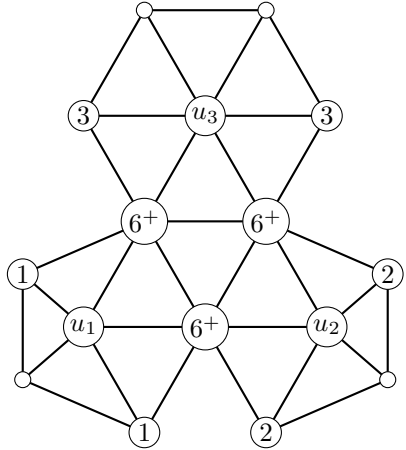
\begin{figure}
\begin{center}
\subfloat[A 3-face $v_1v_2v_3$, such that the pairwise common neighbors of
$v_1$, $v_2$, $v_3$ have degrees 5, 5, and at most 6.
]{\makebox[.45\textwidth]{
\begin{tikzpicture}[xscale = 16, yscale=14]
\tikzstyle{VertexStyle}=[shape = circle, minimum size = 6pt, inner sep = 1.2pt, draw]
\Vertex[x = 0.60, y = 0.85, L = \small {$u_3$}]{v0}
\Vertex[x = 0.55, y = 0.75, L = \small {$6^+$}]{v1}
\Vertex[x = 0.65, y = 0.75, L = \small {$6^+$}]{v2}
\Vertex[x = 0.60, y = 0.65, L = \small {$6^+$}]{v3}
\Vertex[x = 0.70, y = 0.65, L = \small {$u_2$}]{v4}
\Vertex[x = 0.50, y = 0.65, L = \small {$u_1$}]{v5}
\Vertex[x = 0.75, y = 0.70, L = \small {$2$}]{v6}
\Vertex[x = 0.65, y = 0.55, L = \small {$2$}]{v7}
\Vertex[x = 0.75, y = 0.60, L = \small {}]{v8}
\Vertex[x = 0.45, y = 0.70, L = \small {$1$}]{v9}
\Vertex[x = 0.45, y = 0.60, L = \small {}]{v10}
\Vertex[x = 0.55, y = 0.55, L = \small {$1$}]{v11}
\Vertex[x = 0.55, y = 0.95, L = \small {}]{v12}
\Vertex[x = 0.65, y = 0.95, L = \small {}]{v13}
\Vertex[x = 0.70, y = 0.85, L = \small {$3$}]{v14}
\Vertex[x = 0.50, y = 0.85, L = \small {$3$}]{v15}
\Edge[](v2)(v0)
\Edge[](v2)(v1)
\Edge[](v2)(v3)
\Edge[](v2)(v4)
\Edge[](v3)(v4)
\Edge[](v3)(v1)
\Edge[](v3)(v5)
\Edge[](v1)(v5)
\Edge[](v1)(v0)
\Edge[](v6)(v4)
\Edge[](v7)(v4)
\Edge[](v8)(v4)
\Edge[](v6)(v2)
\Edge[](v6)(v8)
\Edge[](v7)(v8)
\Edge[](v7)(v3)
\Edge[](v9)(v1)
\Edge[](v9)(v5)
\Edge[](v9)(v10)
\Edge[](v11)(v3)
\Edge[](v11)(v5)
\Edge[](v11)(v10)
\Edge[](v10)(v5)
\Edge[](v0)(v15)
\Edge[](v0)(v12)
\Edge[](v0)(v13)
\Edge[](v0)(v14)
\Edge[](v14)(v2)
\Edge[](v14)(v13)
\Edge[](v12)(v13)
\Edge[](v12)(v15)
\Edge[](v15)(v1)
\end{tikzpicture}}}
\caption{The key case of Lemma~\ref{alpha556tri}.}
\label{figalpha556tri}
\end{center}
\end{figure}
\begin{proof}
If the only pairwise common neighbors of the $u_i$ are the $v_i$, then two
$u_i$ are 5-vertices and the third is a $6^-$-vertex. 
The case where the $u_i$ have more pairwise common neighbors is nearly identical,
and we remark on it briefly at the end of the proof.
So suppose that $d(u_1)=d(u_2)=5$ and $d(u_3)=6$, as shown in
Figure~\ref{figalpha556tri}; the case where $d(u_3)=5$ is nearly identical. 
We will apply Lemma~\ref{ind-red} with $J=\set{u_1,u_2,u_3}$ and $k=0$.
Clearly, $J$ is an independent set.  Now we verify that each vertex
of $J$ satisfies condition (2).  Since $G$ has no separating 3-cycle, the two
vertices in each pair with a common label (among $\set{1,2,3}$) are distinct and
nonadjacent.  Similarly, the vertices with labels in $\set{1,2,3}$ are
distinct, since they are drawn at pairwise distance at most three, and $G$ has no
separating 3-cycle.  Thus, we can apply Lemma~\ref{ind-red}, as desired.

In the more general case where the $u_i$ have pairwise common neighbors in
addition to the $v_i$, the argument above still shows that the vertices with
labels in $\set{1,2,3}$ are distinct.  So again, we can apply
Lemma~\ref{ind-red} with $k=0$.
\end{proof}
\begin{figure}
\begin{center}
\subfloat[Here $u_2$ and $u_3$ have a common neighbor in
$N(u_1)$.]{\makebox[.45\textwidth]{
\begin{tikzpicture}[scale = 15]
\tikzstyle{VertexStyle}=[shape = circle, minimum size = 6pt, inner sep = 1.2pt, draw]
\Vertex[x = 0.55, y = 0.65, L = \small {$u_1$}]{v0}
\Vertex[x = 0.55, y = 0.80, L = \small {}]{v1}
\Vertex[x = 0.45, y = 0.75, L = \small {}]{v2}
\Vertex[x = 0.65, y = 0.75, L = \small {$4$}]{v3}
\Vertex[x = 0.70, y = 0.65, L = \small {$1$}]{v4}
\Vertex[x = 0.65, y = 0.55, L = \small {$4$}]{v5}
\Vertex[x = 0.45, y = 0.55, L = \small {$1$}]{v6}
\Vertex[x = 0.40, y = 0.65, L = \small {$4$}]{v7}
\Vertex[x = 0.65, y = 0.85, L = \small {$u_3$}]{v8}
\Vertex[x = 0.45, y = 0.85, L = \small {$u_2$}]{v9}
\Vertex[x = 0.60, y = 0.95, L = \small {$3$}]{v10}
\Vertex[x = 0.75, y = 0.80, L = \small {$3$}]{v11}
\Vertex[x = 0.75, y = 0.95, L = \small {$4$}]{v12}
\Vertex[x = 0.50, y = 0.95, L = \small {$2$}]{v13}
\Vertex[x = 0.35, y = 0.80, L = \small {$2$}]{v14}
\Vertex[x = 0.35, y = 0.95, L = \small {}]{v15}
\Edge[](v1)(v0)
\Edge[](v2)(v0)
\Edge[](v3)(v0)
\Edge[](v4)(v0)
\Edge[](v5)(v0)
\Edge[](v6)(v0)
\Edge[](v7)(v0)
\Edge[](v1)(v3)
\Edge[](v4)(v3)
\Edge[](v4)(v5)
\Edge[](v6)(v5)
\Edge[](v2)(v7)
\Edge[](v6)(v7)
\Edge[](v1)(v2)
\Edge[](v9)(v13)
\Edge[](v9)(v15)
\Edge[](v9)(v14)
\Edge[](v14)(v15)
\Edge[](v13)(v15)
\Edge[](v8)(v11)
\Edge[](v8)(v12)
\Edge[](v8)(v10)
\Edge[](v10)(v12)
\Edge[](v11)(v12)
\Edge[](v9)(v2)
\Edge[](v9)(v1)
\Edge[](v8)(v1)
\Edge[](v8)(v3)
\Edge[](v11)(v3)
\Edge[](v14)(v2)
\Edge[](v1)(v13)
\Edge[](v1)(v10)
\end{tikzpicture}}}

\subfloat[Here $u_2$ and $u_3$ have adjacent neighbors in
$N(u_1)$.]{\makebox[.45\textwidth]{
\begin{tikzpicture}[scale = 15]
\tikzstyle{VertexStyle}=[shape = circle, minimum size = 6pt, inner sep = 1.2pt, draw]
\Vertex[x = 0.50, y = 0.50, L = \small {$u_1$}]{v0}
\Vertex[x = 0.40, y = 0.40, L = \small {$4b$}]{v1}
\Vertex[x = 0.60, y = 0.40, L = \small {}]{v2}
\Vertex[x = 0.65, y = 0.50, L = \small {$4$}]{v3}
\Vertex[x = 0.60, y = 0.60, L = \small {$1$}]{v4}
\Vertex[x = 0.50, y = 0.65, L = \small {}]{v5}
\Vertex[x = 0.40, y = 0.60, L = \small {$1$}]{v6}
\Vertex[x = 0.35, y = 0.50, L = \small {$4$}]{v7}
\Vertex[x = 0.75, y = 0.45, L = \small {$u_3$}]{v8}
\Vertex[x = 0.25, y = 0.45, L = \small {$u_2$}]{v9}
\Vertex[x = 0.15, y = 0.50, L = \small {$4$}]{v10}
\Vertex[x = 0.25, y = 0.55, L = \small {$2$}]{v11}
\Vertex[x = 0.15, y = 0.40, L = \small {$2$}]{v12}
\Vertex[x = 0.85, y = 0.50, L = \small {$4$}]{v13}
\Vertex[x = 0.85, y = 0.40, L = \small {$3$}]{v14}
\Vertex[x = 0.75, y = 0.55, L = \small {$3$}]{v15}
\Edge[](v1)(v0)
\Edge[](v1)(v7)
\Edge[](v2)(v0)
\Edge[](v2)(v1)
\Edge[](v3)(v0)
\Edge[](v3)(v2)
\Edge[](v3)(v4)
\Edge[](v4)(v0)
\Edge[](v5)(v0)
\Edge[](v5)(v4)
\Edge[](v5)(v6)
\Edge[](v6)(v0)
\Edge[](v6)(v7)
\Edge[](v7)(v0)
\Edge[](v8)(v2)
\Edge[](v8)(v3)
\Edge[](v9)(v7)
\Edge[](v9)(v1)
\Edge[](v11)(v9)
\Edge[](v11)(v7)
\Edge[](v11)(v10)
\Edge[](v12)(v10)
\Edge[](v12)(v9)
\Edge[](v12)(v1)
\Edge[](v9)(v10)
\Edge[](v8)(v14)
\Edge[](v8)(v13)
\Edge[](v8)(v15)
\Edge[](v15)(v3)
\Edge[](v15)(v13)
\Edge[](v14)(v2)
\Edge[](v14)(v13)
\end{tikzpicture}}}

\subfloat[Here $u_2$ and $u_3$ have neighbors at
distance 2\ in $N(u_1)$.]{\makebox[.45\textwidth]{
\begin{tikzpicture}[scale = 15]
\tikzstyle{VertexStyle}=[shape = circle,	
								 minimum size = 6pt,
								 inner sep = 1.2pt,
                                 draw]
\Vertex[x = 0.550000011920929, y = 0.650000005960464, L = \small {$u_1$}]{v0}
\Vertex[x = 0.550000011920929, y = 0.799999997019768, L = \small {$1$}]{v1}
\Vertex[x = 0.449999988079071, y = 0.75, L = \small {$4$}]{v2}
\Vertex[x = 0.649999976158142, y = 0.75, L = \small {$4$}]{v3}
\Vertex[x = 0.699999988079071, y = 0.650000005960464, L = \small {$4c$}]{v4}
\Vertex[x = 0.649999976158142, y = 0.550000011920929, L = \small {}]{v5}
\Vertex[x = 0.449999988079071, y = 0.550000011920929, L = \small {$1$}]{v6}
\Vertex[x = 0.400000005960464, y = 0.650000005960464, L = \small {$4b$}]{v7}
\Vertex[x = 0.75, y = 0.75, L = \small {$u_3$}]{v8}
\Vertex[x = 0.349999994039536, y = 0.75, L = \small {$u_2$}]{v9}
\Vertex[x = 0.699999988079071, y = 0.849999994039536, L = \small {$3$}]{v10}
\Vertex[x = 0.850000023841858, y = 0.699999988079071, L = \small {$3$}]{v11}
\Vertex[x = 0.850000023841858, y = 0.849999994039536, L = \small {$4$}]{v12}
\Vertex[x = 0.400000005960464, y = 0.849999994039536, L = \small {$2$}]{v13}
\Vertex[x = 0.25, y = 0.699999988079071, L = \small {$2$}]{v14}
\Vertex[x = 0.25, y = 0.849999994039536, L = \small {$4$}]{v15}
\Edge[](v1)(v0)
\Edge[](v2)(v0)
\Edge[](v3)(v0)
\Edge[](v4)(v0)
\Edge[](v5)(v0)
\Edge[](v6)(v0)
\Edge[](v7)(v0)
\Edge[](v1)(v3)
\Edge[](v4)(v3)
\Edge[](v4)(v5)
\Edge[](v6)(v5)
\Edge[](v2)(v7)
\Edge[](v6)(v7)
\Edge[](v1)(v2)
\Edge[](v8)(v3)
\Edge[](v8)(v4)
\Edge[](v9)(v2)
\Edge[](v9)(v7)
\Edge[](v9)(v13)
\Edge[](v9)(v15)
\Edge[](v9)(v14)
\Edge[](v14)(v7)
\Edge[](v14)(v15)
\Edge[](v13)(v15)
\Edge[](v13)(v2)
\Edge[](v8)(v11)
\Edge[](v8)(v12)
\Edge[](v8)(v10)
\Edge[](v10)(v3)
\Edge[](v10)(v12)
\Edge[](v11)(v4)
\Edge[](v11)(v12)
\end{tikzpicture}}}

\caption{The cases of Lemma~\ref{alpha755big}. The three possibilities for an
independent 3-set $\set{u_1,u_2,u_3}$ where $d(u_1)=7$, $d(u_2)=d(u_3)=5$, and
each of $u_2$ and $u_3$ has two neighbors in common with $u_1$.
\label{figalpha755big}
}
\end{center}
\end{figure}
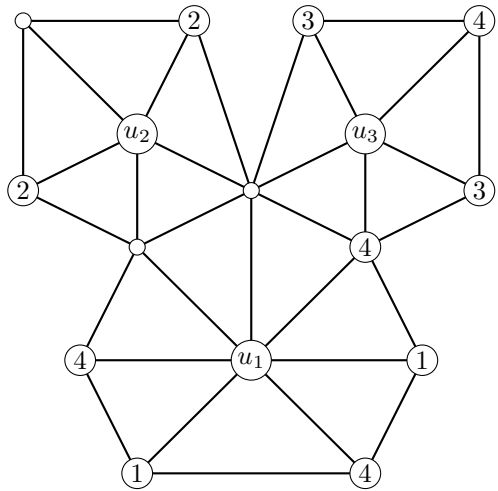
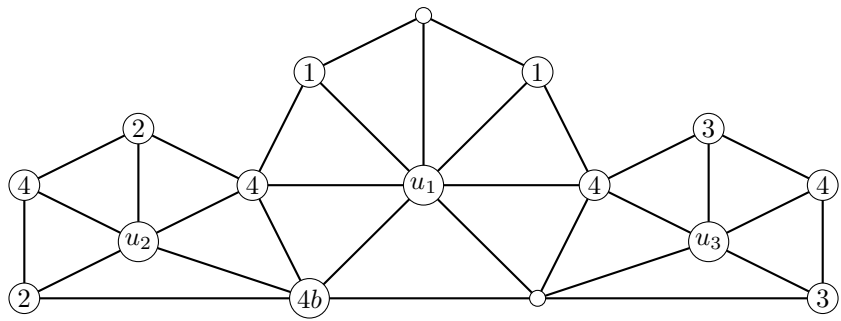
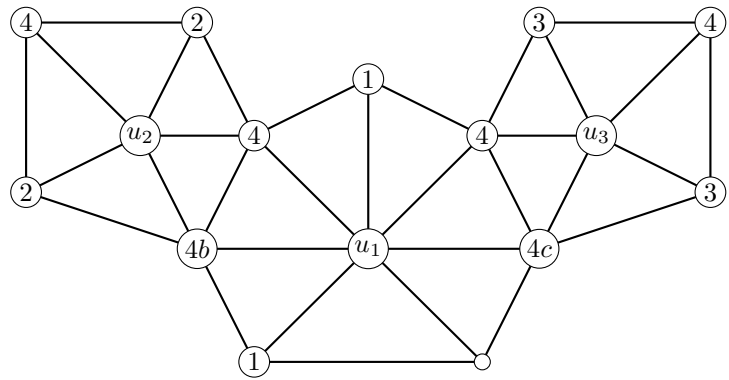

\begin{lem}
Let $u_1$ be a 7-vertex with nonadjacent 5-vertices $u_2$ and $u_3$ each at
distance two from $u_1$.  A minimal $G$ cannot have $u_1$ and $u_2$ with two
common neighbors and also $u_1$ and $u_3$ with two common neighbors.
\label{alpha755big}
\end{lem}
\begin{proof}
This situation is shown in Figures~\ref{figalpha755big}(a,b,c), possibly with
some vertices identified.  Let $J=\{u_1,u_2,u_3\}$.  Suppose that more than a
single pair of vertices is identified, which implies $|N(J)|\le 11$.  If $J$ is
a non-maximal independent set, then this contradicts
Lemma~\ref{CrunchToOneVertexNeighborhoodSize}.  So suppose that $J$ is a
maximal independent set.
If $|N(J)|\le 10$, then $|G|\le 13$, so $J$ is the desired independent set of
size $\frac3{13}|G|$.  Otherwise, $|G|=14$, so exactly three vertices are
identified.  Now we find an independent 4-set.  Either we can take the four
vertices labeled 4, or the two labeled $i$, together with $J\setminus\{u_i\}$,
for some $i\in \{1,2,3\}$.
Thus, at most one pair of vertices drawn as distinct are identified.

If all vertices
labeled $2$ or $3$ are distinct as drawn, then we apply
Lemma~\ref{JIsThreeFourteenWithSeven} and get a contradiction.  By
Lemma~\ref{CrunchToOneVertexNeighborhoodSize}, the only other possibility is
that exactly one pair of vertices is identified. Such a pair must consist of
vertices labeled $2$ and $3$ that are drawn at distance four (otherwise
we apply Lemma~\ref{ind-red}, with $k=1$). In
Figure~\ref{figalpha755big}(a), this is impossible, since
the two $5$-vertices $u_2$ and $u_3$ would have two neighbors in common,
violating Lemma~\ref{alpha55}.

Now we consider the cases shown in Figures~\ref{figalpha755big}(b,c)
simultaneously.  We apply Lemma~\ref{JIsThree} using the vertices labeled $1$
for the independent $2$-set.  Let $I_1$ be the set of vertices labeled $4$. If
$I_1$ is independent, then we are done; so assume not.  Recall that a vertex
labeled $2$ is identified with a vertex labeled $3$. 

Suppose the vertices labeled $4$ in $N(u_2)\setminus N(u_1)$ and $N(u_3)\setminus
N(u_1)$ are not adjacent. Now by symmetry, we may assume that the vertex
labeled $4$ in $N(u_1) \cap N(u_2)$ is adjacent to the vertex
labeled $4$ in $N(u_3) \setminus N(u_1)$. Let $I_2$ be the set made from
$I_1$ by replacing the vertex labeled $4$ in $N(u_1) \cap N(u_2)$ with the vertex
labeled $4b$.  If $I_2$ is independent, then we are done; so assume not.  Now
the vertex labeled $4b$ must be adjacent to the vertex labeled $4$ in $N(u_3)
\setminus N(u_1)$, but this makes a separating 3-cycle (consisting of two
vertices labeled 4 and one labeled $4b$), a contradiction.

So, we may assume that the vertices labeled $4$ in $N(u_2)\setminus N(u_1)$ and
$N(u_3)\setminus N(u_1)$ are adjacent.  Suppose the topmost vertex labeled $2$
is identified with the topmost vertex labeled $3$.  Now again we are done;
our independent 4-set consists of the two neighbors of $u_1$ labeled 4, together
with an independent 2-set from among the two leftmost and two rightmost vertices
(by planarity, they cannot all four be pairwise adjacent).

The only remaining possibility is that the bottommost vertex labeled $2$ is
identified with the bottommost vertex labeled $3$ (since the two topmost
vertices labeled 4 are adjacent).  
If we are in Figure~\ref{figalpha755big}(b), then the vertex labeled $4b$ is
a 5-vertex; since it shares two neighbors with $u_3$, another 5-vertex, we
contradict Lemma~\ref{alpha55}.
Hence, we must be in Figure~\ref{figalpha755big}(c).  
Now our independent 4-set consists of the two neighbors of $u_1$ labeled 4b and
4c, together with an indpendent 2-set from among the four topmost vertices
(again, by planarity, they cannot all be pairwise adjacent).
\end{proof}

\begin{lem}\label{reduce7vertex}
Suppose that a minimal $G$ contains a 7-vertex $v$ with no 5-neighbor.  Now $v$ cannot
have at least five 6-neighbors, each of which has a 5-neighbor. 
\end{lem}
\begin{proof}
Suppose to the contrary.  Denote the neighbors of $v$ in clockwise order by
$u_1,\ldots, u_7$.  

\textbf{Case 1:} Vertices $u_1, u_2, u_3, u_4$ are 6-vertices, each with
a 5-neighbor.  

First, suppose that $u_2$ and $u_3$ have a common 5-neighbor, $w_2$. 
Consider the 5-neighbor $w_1$ of $u_1$.  By Lemma~\ref{alpha55}, it cannot be
common with $u_2$; similarly, the 5-neighbor $w_4$ of $u_4$ cannot be common
with $u_3$.  (We must have $w_1$ and $w_4$ distinct, since otherwise we apply
Lemma~\ref{alpha556tri} to $\set{u_1,u_4,w_2}$.  Also, we must have $w_1$ and
$w_4$ each distinct from $w_2$, since $G$ has no separating 3-cycles.)  

First, suppose that $w_1$ has two common neighbors with $u_2$.
If $w_1\nonadj u_4$, then we apply Lemma~\ref{alpha665big} to
$\set{w_1,u_2,u_4}$; so assume
$w_1\adj u_4$.  Now let $J=\set{u_1,u_4,w_2}$.  Clearly, $J$ is an independent
3-set.  Also $|N(J)|\le 6+6+5-4=13$, so we are done by Lemma~\ref{alpha556tri}.
So $w_1$ cannot have two common neighbors with $u_2$.  Similarly, $w_4$ cannot
have two common neighbors with $u_3$.  Hence, $w_1\adj u_7$ and also
$w_4\adj u_5$.  Now we must have $w_1\adj w_4$; otherwise we apply
Lemma~\ref{alpha755big} to $\set{v,w_1,w_4}$.  
Similarly, we must have $w_1\adj w_2$ and $w_2\adj w_4$; these edges cut off
$w_4$ from $u_1$, so $u_1\nonadj w_4$.
Since $u_1$ and $w_4$ are
nonadjacent, but have a 5-neighbor in common, they must have two neighbors in
common.  So we apply Lemma~\ref{alpha665big} to $\set{u_1,u_3,w_4}$.
Hence, we conclude that the common neighbor of $u_2$ and $u_3$ is not a
5-neighbor.  

Since $u_1$ and $u_3$ are $6^-$-vertices, 
by Lemma~\ref{alpha666}, vertex $u_2$ cannot have another $6^-$-vertex that is
nonadjacent to $u_1$ and $u_3$.  Thus, a 5-neighbor of $u_2$ must be a common
neighbor  of $u_1$; call this $5$-neighbor $w_1$.
Similarly, the common neighbor $w_4$ of $u_3$ and $u_4$ is
a 5-vertex.  We must have $w_1\adj w_4$, for otherwise we apply
Lemma~\ref{alpha755big}.  We may assume that $u_6$ is a 6-vertex.  If not, then
$v$'s five 6-neighbors, each with a 5-neighbor, are \emph{successive}; so, by
symmetry, we are in the case above, where $u_2$ and $u_3$ have a common
5-neighbor.

By planarity, either $u_1\nonadj w_4$ or else $u_4\nonadj w_1$; by symmetry,
assume the former.  Since $u_1$ and $w_4$ share a 5-neighbor (and are
nonadjacent), they have two common neighbors.  Now if $u_6\nonadj w_4$, then we
apply Lemma~\ref{alpha665big} to $\set{u_1,u_6,w_4}$.  Hence, assume $u_6\adj
w_4$.  This implies that $u_4\nonadj w_1$.  Now, the same argument implies that
$u_6\adj w_1$.  Now let $J=\set{u_1,u_4,u_6}$.  
Lemma~\ref{CrunchToOneVertexNeighborhoodSize} gives $12\le |N(J)|\le
6+6+6-6=12$.  Thus the vertices of $J$ have no additional pairwise common
neighbors.  Hence, we have an independent 2-set $M_1$ in 
$N(u_1)\setminus (N(u_4)\cup N(u_6))$.  Similarly, we have an independent 2-set
$M_4$ in $N(u_4)\setminus (N(u_1)\cup N(u_6))$.  Now we apply
Lemma~\ref{ThreeHelper} with $J=\set{u_1,u_4,u_6}$ and $S_1=M_1\cup\set{u_1}$
and $S_2=M_4\cup\set{u_4}$.  In each case, we have $\alpha(G[S_i\cup J])\ge
|M_i\cup\set{u_{5-i},u_6}|=4$.  This implies that $|N(J)|\ge 13$, a
contradiction.  Hence, $v$ cannot have four successive
6-neighbors, each with a 5-neighbor.

\textbf{Case 2:} 
Vertices $u_1, u_2, u_3, u_5, u_6$ are 6-vertices, each with a
5-neighbor.

Suppose that the common neighbor $w_5$ of $u_5$ and $u_6$ is a 5-vertex.
By symmetry (between $u_1$ and $u_3$) and Lemma~\ref{alpha666}, assume that the
common neighbor $w_2$ of
$u_2$ and $u_3$ is a 5-vertex.  If $w_2\nonadj w_5$, then we apply
Lemma~\ref{alpha755big}; so assume that $w_2\adj w_5$.  If $u_6\nonadj w_2$,
then apply Lemma~\ref{alpha665big} to $\set{u_6,u_1,w_2}$ (note that $u_6$ and
$w_2$ have two common neighbors, since they have a common 5-neighbor).  So assume
that $u_6\adj w_2$.  Similarly, we assume that $u_3\adj w_5$, since otherwise we
apply Lemma~\ref{alpha665big} to $\set{u_3,u_1,w_5}$.  Now consider the
5-neighbor $w_1$ of $u_1$.  By Lemma~\ref{alpha55}, it cannot be a common
neighbor of $u_2$ (because of $w_2$).  If it is a common neighbor of $u_7$, then
we apply Lemma~\ref{alpha755big} to $\set{w_1,w_5,v}$; note that $w_1\nonadj
w_5$, since they are cut off by edge $w_2u_6$.  Hence, $w_1$ is neither a common
neighbor of $u_7$ nor of $u_2$.  Now we apply Lemma~\ref{alpha665big} to
$\set{u_2,w_1,w_5}$.  Thus, we conclude that the common neighbor of $u_5$ and
$u_6$ is not a 5-vertex.

Let $x$ denote the common neighbor of $u_5$ and $u_6$; as shown in the previous
paragraph, $x$ must be a $6^+$-vertex.  Suppose that the 5-neighbor $w_5$ of
$u_5$ is also a neighbor of $x$.  If $w_5\nonadj u_1$, then we apply
Lemma~\ref{alpha665big} to $\set{u_6,u_1,w_5}$; so assume that $w_5\adj u_1$.
Now if the 5-neighbor $w_6$ of $u_6$ is also adjacent to $x$, then we apply
Lemma~\ref{alpha665big} to $\set{u_5,w_6,u_3}$ (we must have $w_6\nonadj u_3$
due to edge $w_5u_1$).  So, by symmetry (between $u_5$ and $u_6$), we may assume that
$w_5\adj u_4$.  Now, by Lemma~\ref{alpha666}, the 5-neighbor $w_2$ of $u_2$ has
is adjacent to either $u_1$ or $u_3$.  In either case, we must have
$w_2\adj w_5$; otherwise, we apply Lemma~\ref{alpha755big} to $\set{v,w_2,w_5}$.
If $w_6\adj u_7$, then $w_6\adj w_2$ and $w_6\adj w_5$; otherwise, we apply
Lemma~\ref{alpha755big} to $\set{v,w_6,w_2}$ or $\set{v,w_6,w_5}$.  Now we apply
Lemma~\ref{alpha665big} to $\set{u_5,u_3,w_6}$.  So instead $w_6\nonadj u_7$.
Finally, we apply Lemma~\ref{alpha665big} to $\set{u_5,w_6,u_3}$.
This completes the proof.
\end{proof}

\section*{Acknowledgments}
As we mentioned in the introduction, the ideas in this paper come largely from
Albertson's proof~\cite{Albertson} that planar graphs have independence ratio at least $\frac29$.
In fact, many of the reducible configurations that we use here are special cases
of the reducible configurations in that proof.  We very much like that
paper, and so it was a pleasure to be able to extend Albertson's work.
It seems that the part of his own proof that Albertson was least pleased with
was verifying ``unavoidability'', i.e., showing that every planar graph contains a
reducible configuration.  In the introduction to Albertson~\cite{Albertson}, he
wrote: ``Finally Section 4 is devoted to a massive, ugly edge counting which
demonstrates that every triangulation of the plane must contain some
forbidden subgraph.''  This is one reason that we included in this paper
a short proof of this same unavoidability statement, via discharging.  We think
Mike might have liked it.

The first author thanks his Lord and Savior, Jesus Christ.

\bibliographystyle{abbrvplain}
\bibliography{45ct}
\end{document}